\documentclass{amsart}
\usepackage{amsmath}
\usepackage{amsthm}
\usepackage{amssymb}
\usepackage{graphics}
\usepackage{bm}

\usepackage{amsfonts}
\usepackage{latexsym}
\usepackage{mathrsfs}

\usepackage{color}

\newcommand{\be}{\begin{equation}}
\newcommand{\ee}{\end{equation}}
\newcommand{\ba}{\begin{eqnarray}}
\newcommand{\ea}{\end{eqnarray}}
\newcommand{\bi}{\begin{itemize}}
\newcommand{\ei}{\end{itemize}}
\newcommand{\bn}{\begin{enumerate}}
\newcommand{\en}{\end{enumerate}}
\newcommand{\bbm}{\begin{bmatrix}}
\newcommand{\ebm}{\end{bmatrix}}
\newcommand{\bpm}{\begin{pmatrix}}
\newcommand{\epm}{\end{pmatrix}}
\newcommand{\bp}{\begin{proof}}
\newcommand{\ep}{\end{proof}}
\newcommand{\nn}{\nonumber}

\newcommand{\mr}{\ensuremath{\mathrm}}
\newcommand{\scr}{\ensuremath{\mathscr}}
\newcommand{\mbf}{\ensuremath{\mathbf}}
\newcommand{\mc}{\ensuremath{\mathcal}}

\newcommand{\ov}{\ensuremath{\overline}}
\newcommand{\sm}{\ensuremath{\setminus}}
\newcommand{\wt}{\ensuremath{\widetilde}}

\newcommand{\ga}{\ensuremath{\gamma}}
\newcommand{\Om}{\ensuremath{\Omega}}

\newcommand{\la}{\ensuremath{\lambda }}

\def\C{\mathbb{C}}

\def\D{\mathbb{D}}
\def\T{\mathbb{T}}

\def\N{\mathbb{N}}
\def\B{\mathbb{B}}
\def\A{\mathcal{A}}

\renewcommand{\H}{\ensuremath{\mathcal{H} }}
\newcommand{\J}{\ensuremath{\mathcal{J} }}

\newcommand{\K}{\ensuremath{\mathcal{K} }}
\renewcommand{\L}{\ensuremath{\mathcal{L} }}

\newcommand{\F}{\ensuremath{\mathbb{F} }}

\def\kz{K \{ Z , y , v \}}
\def\kw{K \{ W , x , u \}}
\newcommand{\zyv}{\ensuremath{\left\{ Z,y, v \right\} }}
\newcommand{\wxu}{\ensuremath{\left\{ W,x,u \right\} }}

\def\id{\mathrm{id}}

\newcommand{\ip}[2]{\ensuremath{\left\langle {#1} , {#2} \right\rangle}}
\newcommand{\ipcn}[2]{\ensuremath{\left( {#1} , {#2} \right) _{\C ^n}}}
\newcommand{\ipcm}[2]{\ensuremath{\left( {#1} , {#2} \right) _{\C ^m}}}

\newcommand{\dom}[1]{\ensuremath{\mathrm{Dom} ({#1}) }}

\newcommand{\ran}[1]{\ensuremath{\mathrm{Ran} \left( {#1} \right) }}
\renewcommand{\ker}[1]{\ensuremath{\mathrm{Ker} \left( {#1} \right) }}

\newcommand{\im}[1]{\ensuremath{\mathrm{Im} \left( {#1} \right) }}
\newcommand{\re}[1]{\ensuremath{\mathrm{Re} \left( {#1} \right) }}

\numberwithin{equation}{section}
\numberwithin{subsection}{section}

\newtheorem{thm}[subsection]{Theorem}
\newtheorem*{thm*}{Theorem}
\newtheorem{claim}[subsection]{Claim}

\newtheorem{lemma}[subsection]{Lemma}
\newtheorem{prop}[subsection]{Proposition}
\newtheorem{cor}[subsection]{Corollary}

\theoremstyle{definition}
\newtheorem{defn}[subsection]{Definition}
\newtheorem{remark}[subsection]{Remark}

\title{Column extreme multipliers of the Free Hardy space}

\author{Michael T. Jury}
\address{University of Florida}
\email{mjury@ufl.edu}

\author{Robert T.W. Martin}
\address{University of Cape Town}
\email{rtwmartin@gmail.com}

\begin{document}
\small
\date{\today}
\bibliographystyle{plain}
\maketitle

\begin{abstract}
The full Fock space over $\C ^d$ can be identified with the free Hardy space, $H^2 (\B ^d _\N)$ - the unique non-commutative reproducing kernel Hilbert space corresponding to a non-commutative Szeg\"{o} kernel on the non-commutative, multi-variable open unit ball $\B ^d _\N := \bigsqcup _{n=1} ^\infty \left( \C^{n\times n} \otimes \C ^d \right) _1$. 

Elements of this space are free or non-commutative functions on $\B ^d _\N$. Under this identification, the full Fock space is the canonical non-commutative and several-variable analogue of the classical Hardy space of the disk, and many classical function theory results have faithful extensions to this setting.  In particular to each contractive (free) multiplier $B$ of the free Hardy space, we associate a Hilbert space $\mathcal H(B)$ analogous to the deBranges-Rovnyak spaces in the unit disk, and consider the ways in which various properties of the free function $B$ are reflected in the Hilbert space $\mathcal H(B)$ and the operators which act on it. In the classical setting, the $\mathcal H(b)$ spaces of analytic functions on the disk display strikingly different behavior depending on whether or not the function $b$ is an extreme point in the unit ball of $H^\infty(\mathbb D)$. We show that such a dichotomy persists in the free case, where the split depends on whtether or not $B$ is what we call {\em column extreme}. 
\end{abstract}

\section{Introduction}

The classical Hardy space, $H^2 (\D)$, can be defined as the Hilbert space of all analytic functions on $\D$ whose Taylor series at $0$ have square summable coefficients (and with inner product equal to the $\ell ^2$ inner product of these Taylor coefficients).  Equivalently, $H^2 (\D) = \H (k)$, is the unique reproducing kernel Hilbert space (RKHS) of functions on $\D$ corresponding to the positive sesqui-analytic kernel function $k: \D \times \D \rightarrow \C$:
$$ k (z,w):= \frac{1}{1-zw^*}; \quad \quad z,w \in \D,$$ the \emph{Szeg\"{o} kernel}. The operator of multiplication by $z$ on $H^2 (\D)$ is called the \emph{shift}, and it is easily seen to be isomorphic to the unilateral shift on $\ell ^2 (\N _0)$, where $\N _0$ denotes the non-negative integers. Proofs of many deep results in classical Hardy space theory ultimately appeal to the fact that $S$ is the universal cyclic pure isometry (recall the Wold decomposition says that any isometry is unitarily equivalent to a direct sum of shifts and a unitary operator). 

From the viewpoint of reproducing kernel theory and operator theory, the canonical (commutative) multi-variable analogue of the Hardy space is then the Drury-Arveson space, $H^2 _d := \mc{H} (k)$, where now $k: \B ^d \times \B ^d \rightarrow \C$ is:
$$  k (z,w):= \frac{1}{1-zw^*}; \quad \quad z,w \in \B^d,$$ the multi-variable \emph{Szeg\"{o} kernel}, 
and $zw^* := z_1 w_1 ^* + ... z_d w_d ^* = \left( w ,z \right) _{\C ^d }$. (Here, $\B ^d := (\C ^d ) _1$, the multi-variable open unit ball.) The appropriate analogue of the shift in this setting is the Arveson $d-$shift, $S = (S_1, ... ,S_d) : H^2 _d \otimes \C ^d \rightarrow H^2 _d$, $(S_j h) (z) := z_j h(z); \ z=(z_1,..., z_d) \in \B ^d$. This is a (row) partial isometry (from $d$ copies of $H^2 _d$ into itself), but no longer an isometry, and this defect is the source of several differences between the single and several-variable theories. Faithful analogues of classical Hardy space results typically seem to exist, but often new (and often more complicated) proof techniques and approaches are required \cite{Sha2013}. 

An alternative approach to extending Hardy space theory from one to several variables would be to seek analogues of Hardy space results for a several-variable shift.  Namely, the natural multi-variable analogue of $\ell ^2 (\N _0 )$ is $\ell ^2 (\F ^d )$, where $\F ^d$ is the free monoid (unital semi-group) of all words in the $d$ letters $\{ 1, ... , d \}$, and with unit equal to the empty word, $\emptyset$, containing no letters. This monoid can be identified with a simple directed tree starting at single node and with $d$ branches at each node (clearly $\F ^1 \simeq \N _0$). There is a natural $d-$tuple of shifts, $L := (L_1 , ... , L_d )$ on $\ell ^2 (\F ^d)$ which are defined by
$$ L _k e _\alpha := e_{k\alpha}, $$ where $\{ e_\alpha \} _{\alpha \in \F ^d}$ is the canonical orthonormal basis of $\ell ^2 (\F ^d )$. It is easy to see that each $L_k$ is a pure isometry and the $L_k$ have pairwise orthogonal ranges $L_k ^* L_j = \delta _{k,j} I$. In particular, the row
$L = (L_1 ,... ,L _d) : \ell ^2 (\F ^d ) \otimes \C ^d \rightarrow \ell ^2 (\F ^d)$ is an isometry from $d$ copies of $\ell ^2 (\F ^d)$ into itself which we call the \emph{left free shift}. The Popescu-Wold decomposition for row isometries shows that $L$ has the same universal property as the shift $S$: any row isometry (an isometry from $d$ copies of a Hilbert space into itself) is isomorphic to the direct sum of several copies of $L$ and a \emph{row unitary} (an onto row isometry). 

The left free shifts $L_k$ are of course non-commuting, and it would appear that one loses the analytic function theory interpretation of the shift as acting as multiplication by the independent variable on a space of analytic functions. Surprisingly, this is not the case: the fields of non-commutative function theory \cite{KVV,Ag-Mc,Pop-freeholo,Pop-freeholo2,Pop-freehar}, and the recently developed theory of non-commutative reproducing kernel Hilbert spaces (NC-RKHS) \cite{BMV} have shown that $\ell ^2 (\F ^d)$ is canonically isomorphic to the \emph{free Hardy space}, $H^2 (\B ^d _\N)$ of non-commutative or \emph{free holomorphic functions} on a certain \emph{non-commutative multi-variable open unit ball}, $\B ^d _\N$ (we will introduce these objects and this theory in an upcoming subsection).

The Drury-Arveson space $H^2 _d$ can be identified with a subspace of $H^2 (\B ^d _\N)$ which is co-invariant and cyclic for both the left and right free shifts:
$$ H^2 _d \simeq \bigvee _{z \in \B^d} K _z \subseteq H^2 (\B ^d _\N ), $$ the span of all the kernel vectors at level one. This subspace is the orthogonal complement of the range of both a right inner and a left inner free multiplier. For example, if $d=2$, 
$$ H^2 _d \simeq \ran{\frac{1}{\sqrt{2}} (L_1 L_2 - L_2 L_1) } ^\perp, $$ and this shows that the theory of $H^2 _d$ should be closer in analogy to that of the theory of model subspaces of $H^2 (\D)$.  In particular, commutative Drury-Arveson space analogues of all of the results of this paper (and those of \cite{JMfree}) can be easily obtained by compression.

In recent work, we have extended Hardy space results including the concept of Aleksandrov-Clark measure, the theory of Clark's unitary perturbations, and equivalent characterizations of extreme points from one to several variables. In particular, the reference \cite{JMfree} extends the theory of Clark measures and Clark peturbations to the non-commutative setting of the full Fock space over $\C ^d$ (which can be identified with $\ell ^2 (\F ^d )$) using the theory of free formal reproducing kernel Hilbert spaces \cite{Ball2003rkhs}.  

The goal of this paper is to develop non-commuative analogues of our recent results on extreme points of the closed unit ball of the multiplier algebra of Drury-Arveson space \cite{JM,JMqe}.  We will also extend and re-cast the main results of \cite{JMfree} in the modern language of NC-RKHS.  In particular we give a number of equivalent characterizations of so-called {\em column extreme} multipliers of the free Hardy space.

\section{Preliminaries} \label{prelim}
All Hilbert space inner products will be conjugate linear in their first argument. If $X$ is a Banach space, $(X)_1$ and $[X]_1$ denote the open and closed unit balls of $X$, respectively.

\subsection{The full Fock space}

Recall that the full Fock space over $\C ^d$, $F^2 _d$, is the direct sum of all tensor powers of $\C ^d$:
\ba F^2 _d & := & \C \oplus \left( \C ^d \otimes \C ^d \right) \oplus \left( \C ^d \otimes \C ^d \otimes \C ^d \right) \oplus \cdots \nn \\
& =& \bigoplus _{k=0} ^\infty \left( \C ^d \right) ^{k \cdot \otimes }. \nn \ea Fix an orthonormal basis $\{ e_1, ... , e_d \}$ of $\C ^d$. The left creation
operators $L_1, ..., L_d$ are the operators which act as tensoring on the left by these basis vectors:
$$ L_k f := e_k \otimes f; \quad \quad f \in F^2 _d, $$ and similarly the right creation operators $R_k; \ 1\leq k \leq d$ are defined by tensoring on the right $$ R_k f := f \otimes e_k. $$ The left and right free shifts are the row operators $L := (L_1 , ... , L_d)$ and $R := (R_1, ... , R _d )$ which map
$F^2 _d \otimes \C ^d$ into $F^2 _d$. Both $L, R$ are in fact row isometries: $L ^* L = I_{F^2} \otimes I_d = R^* R$. It follows that the component shifts are also isometries with pairwise orthogonal ranges.  The orthogonal complement of the range of $L$ or $R$ is the vacuum vector $1$ which spans the the subspace $\C =: (\C^d) ^{0\cdot \otimes} \subset F^2 _d$. A canonical orthonormal basis for $F^2 _d$ is then $\{ e_\alpha \} _{\alpha \in \F ^d}$ where $e_\alpha = L^\alpha 1 = R^\alpha 1$ and $\F ^d$ is the free unital semigroup or monoid on $d$ letters. Here, if $\alpha = i_1 \cdots i_n \in \F ^d$, we use the standard notation $L^\alpha = L_{i_1} L_{i_2} \cdots L_{i_n}$.

Recall here that the free monoid, $\F^d$, on $d \in \N$ letters, is the multiplicative semigroup of all finite products or \emph{words} in
the $d$ letters $\{1, ... , d \}$. That is, given words $\alpha := i_1 ... i_n$, $\beta := j_1 ... j_m$, $i_k, j_l \in \{1 , ... , d \}; \ 1\leq k \leq n, \ 1 \leq l \leq m$, their product $\alpha \beta $ is defined by concatenation:
$$ \alpha \beta = i_1 ... i_nj_1 ... j_m, $$ and the unit is the empty word, $\emptyset$, containing no letters. Given $\alpha = i_1 \cdots i_n$, we use the standard notation $|\alpha | = n$ for the length of the word $\alpha$. The transpose map $\dag : \F ^d \rightarrow \F ^d$, defined by 
$$  i_1 \cdots i_d =\alpha \mapsto \alpha ^\dag := i_d \cdots i_1,  \quad \quad \mbox{is an involution.} $$ 

Define $L^\infty _d := \mr{Alg}(I,L) ^{-WOT}, \ R^\infty _d := \mr{Alg}(I,R) ^{-WOT}$, the left (\emph{resp.} right) free analytic Toeplitz algebra ($WOT$ denotes weak operator topology).  The \emph{transpose unitary}, $U _{\dag} : F^2 _d \rightarrow F^2 _d$, defined by $e_{\alpha } \mapsto e_{\alpha ^{\dag}}$ is a unitary involution of $F^2 _d$, and it is easy to verify that
$$ U_\dag L_k U_\dag^* = R_k, $$ so that adjunction by $U_\dag$ implements a unitary isomorphism between $L^\infty _d$ and $R^\infty _d$.

\subsection{The free Hardy space}

It will be convenient to view $F^2 _d$ as a non-commutative reproducing kernel Hilbert space (NC-RKHS) \cite{BMV} of freely non-commutative (holomorphic) functions on the non-commutative open unit ball \cite{KVV}:
$$ \B _\N ^d := \bigsqcup _{n=1} ^\infty \B ^d_n ; \quad \quad \B ^d_n   := \left( \C ^{n \times n} \otimes \C ^d \right) _1. $$ Elements of $\B ^d_n $ are viewed as strict row contractions on $\C ^n$.
Recall that for any complex vector space $V$,
$$ V_{nc} := \bigsqcup V_n; \quad \quad V_n := V \otimes \C ^{n\times n} =: V^{n \times n}. $$
The NC unit ball $\B ^d _\N$ is an example of a NC set: A set $\Om \subseteq V _{nc}$ is an NC set if it is closed under direct sums, and one writes:
$$ \Om =: \bigsqcup \Om _n; \quad \quad \Om _n := \Om \cap V_n. $$ A function $f : \Om \rightarrow \C _{nc}$ is called a NC or \emph{free function} if: 
$$ f : \Om _n \rightarrow \C ^{n\times n}; \quad \quad f \ \mbox{respects the grading}, $$ and
if $X \in \C ^{n\times m }, Z \in X_n , W \in X _m$ obey $ZX = XW$, then, 
$$ f(Z) X = X f(W); \quad \quad f \ \mbox{respects intertwinings.}$$

As shown in \cite{BMV}, $F^2 _d =  H^2 (\B _\N ^d )$ can be viewed as the \emph{free Hardy space} of the multi-variable NC unit ball $\B _\N ^d$, \emph{i.e.}
$H^2 (\B _\N ^d) = \H _{nc} (K)$ is the unique NC-RKHS corresponding to the NC-Szeg\"{o} kernel:
$K : \B _\N ^d \times \B _\N ^d \rightarrow  \L (\C _{nc} )$ defined by:
$$ K (Z, W)  [P ] := \sum _{\alpha \in \F ^d} Z^\alpha P (W^* ) ^{\alpha ^\dag}; \quad \quad Z \in \B ^d_n  , W \in \B ^d_m , P \in \C ^{n\times m}.$$
See \cite{BMV} for the full definition and theory of NC kernels. In particular, any NC kernel respects the grading and intertwinings in both arguments \cite[Section 2.3]{BMV}.

One can show that elements of $H^2 (\B _\N ^d ) := \H _{nc} (K)$ are locally bounded (and hence automatically) holomorphic free functions on $\B ^d _\N$ \cite[Chapter 7]{KVV}.  That is, any $f\in H^2 (\B _\N ^d )$ is Fr\'{e}chet and  G\^{a}teaux differentiable at any point $Z \in \B ^d _\N$ and $f$ has a convergent power series expansion (Taylor-Taylor series) about any point. 

(Generally any) $\H _{nc} (K)$ is formally defined as the Hilbert space completion of the linear span:
$$ \bigvee _{Z \in \B ^d _n, \ y,v \in \C ^n } \kz, $$ where the $\kw$ are the free functions on $\B ^d _\N$, $ \kw : \B ^d _n \rightarrow \C ^{n\times n}$, defined by:
$$ \kw (Z) y := K(Z,W) [yu^*] x; \quad \quad W \in \B ^d _m, \ Z \in \B ^d _n; \ u,x \in \C ^m, \ y \in \C ^n.$$ Completion is with respect to the inner product:
\begin{align*} & \ip{\kz}{\kw}  :=   \ipcn{y}{K(Z,W)[vu^*]x}; \\
& Z \in \B ^d_n , \ v , y \in \C ^n; \ W \in \B ^d _m, \ u,x \in \C ^m. \end{align*}

These point evaluation vectors have a familiar reproducing property: $K(Z,y,v)$ is the unique vector in $H^2 (\B _\N ^d )$ such that for any $f \in H^2 (\B _\N ^d )$,
\begin{equation}\label{eqn:kz-reproducing-formula}
\ip{\kz }{f}  = \ipcn{y}{f(Z)v }. 
\end{equation}

For any $Z \in \B ^d_n$ one can also define a natural \emph{kernel map} $K_Z \in \L (\C ^{n \times n} , H^2 (\B ^d _\N) )$ as follows: Any $A \in \C ^{n \times n}$
can be written as a linear combination of the rank one outer products
$$ y   v^*  = \bbm y_1\\ \vdots \\ y_n  \ebm \bbm \overline{v_1}, \ & \cdots \ & , \overline{v_n} \ebm ; \quad \quad y \in \C ^n, v^* \in (\C^n)^*. $$
Then we define $K_Z$ on rank one matrices $yv^*$ by the formula
\begin{equation}\label{eqn:K-bigZ-def}
 K_Z ( yv ^*) := \kz \in H^2 (\B _\N ^d ). 
\end{equation}
 Let us check that $K_Z$ is well defined: the vectors $y$ and $v$ determining a rank one matrix $yv^*$ are unique up to the scaling $y\to \lambda y, v\to \overline{\lambda}^{-1}v$ where $\lambda$ is any nonzero complex number. From the reproducing formula (\ref{eqn:kz-reproducing-formula}), it is evident that the vector $K\{Z,y,v\}$ is invariant under such a scaling, and so the formula (\ref{eqn:K-bigZ-def}) is unambiguous.  If we view $\C ^{n \times n}$ as a Hilbert space equipped with the normalized trace inner product, then $K_Z : \C ^{n \times n } \rightarrow H^2 (\B ^d _\N )$ extends to a bounded linear map, and its Hilbert space adjoint is the point evaluation map at $Z$:
$$ K_Z ^* F = F(Z) \in \C ^{n \times n}. $$ 

The free Hardy space and the full Fock space are canonically isomorphic: Define $\scr{U} : F^2 _d \rightarrow  H^2 (\B ^d _\N ) $ by:
\begin{align*} x & :=  \sum _{\alpha \in \F ^d } x_\alpha L^\alpha 1 \stackrel{\scr{U}}{\mapsto} f_x \in H^2 (\B _\N ^d ), \\
 f_x (Z) &:=  \sum _{\alpha \in \F ^d} Z^\alpha x_\alpha; \quad \quad Z \in \B _\N ^d. \end{align*}
The inverse, $\scr{U} ^{-1}$, acts on kernel vectors as:
 \be \kz \stackrel{\scr{U} ^{-1} }{\mapsto} x[Z,y,v] := \sum _{\alpha \in \F ^d} \ip{Z^\alpha v}{y} L^\alpha 1 \in F ^2 _d. \label{pevalimg}\ee

\subsection{Left and Right free multipliers}
\label{leftrightmult}
As in the classical setting, given a NC-RKHS $\H _{nc} (K )$ on an NC set $\Om$ (\emph{e.g.} $\B ^d _\N$), it is natural to consider the left and right \emph{multiplier algebras} $$ \mr{Mult} ^L (\H _{nc} (K) ), \ \mr{Mult} ^R (\H _{nc} (K) )$$ of NC functions on $\Om$ which left or (\emph{resp.}) right multiply $\H _{nc} (K)$ into itself. 
Namely, a free function $F$ on $\B _\N ^d$ is said to be a \emph{left free multiplier} if, for any $f \in H^2 (\B _\N ^d )$, $Ff \in H^2 (\B ^d _\N)$. Similarly, $G$ is called a \emph{right free multiplier}
if $fG \in H^2 (\B _\N ^d )$ for all $f \in H^2 (\B _\N ^d )$. As in the classical setting, the left and right \emph{free multiplier algebras}, $H^\infty _L (\B _\N ^d ) := \mr{Mult} ^L (H^2 (\B ^d _\N ) ), H^\infty _R (\B _\N ^d )$ are weak operator toplogy (WOT)-closed unital operator algebras. Moreover, adjunction by the canonical unitary $\scr{U}$ defines a unitary $*-$isomorphism of the left and right free analytic Toeplitz algebras $L^\infty _d , R^\infty _d$ onto these
left and right free multiplier algebras of $H^2 (\B _\N ^d )$.  As in the classical setting of $H^\infty (\D)$, the multiplier norm of any $F \in H^\infty _L (\B _\N ^d )$ can be computed as the supremum norm on the NC unit ball:
$$ \| F \| := \sup _{Z \in \B _\N ^d} \| F (Z ) \|. $$ The left and right Schur classes, $\scr{L} _d , \scr{R} _d$ are then defined as the closed unit balls of these left and right multiplier algebras (equivalently as the closed unit balls of $L^\infty _d , R^\infty _d$). 

Observe that if $F$ is a left free multiplier then,
\ba \ip{\kz}{Ff} & = & \ip{y}{F(Z) f(Z) v} \nn \\
& = & \ip{K \{ Z, F(Z) ^* y, v \} }{f}, \nn \ea so that
\be (M^L _F)^* \kz = K \{ Z,F(Z) ^* y,v \}, \label{leftpevalue} \ee and similarly, if $G$ is a right free multiplier,
\be (M^R _G)^* \kz = K \{Z ,y,G(Z) v \}. \label{rightpevalue} \ee Alternatively, using the kernel maps $K_Z$, we can write:
$$ (M_F ^L) ^* K_Z (yv) = K_W (F(Z)^*yv^*), \quad \quad \mbox{and} \quad \quad (M_G^R)^* K_Z = K_Z  (y v^*G(Z)^*). $$
One can check that if, \emph{e.g.}, right multiplication by $G(Z)$ is a right free multiplier then
$$(   (M_G^R)^* K_Z)^* ((M_G^R)^* K_W ) = K(Z,W)[G(Z) \cdot G(W) ^*]. $$

In particular, free holomorphic $F(Z),G (Z)$ belong to the left or right Schur classes if and only if
$$ K^F (Z,W)[\cdot]:= K(Z,W) - F(Z) K(Z,W)[\cdot] F(W) ^*  $$ or
$$ K^G (Z,W) [\cdot]:= K(Z,W) - K(Z,W)[G(Z) [\cdot] G(W) ^* ]  $$ are CPNC kernels, respectively. These NC kernels are called the left or right free deBranges-Rovnyak kernels of $F, G$ (\emph{resp.}) and in this case the corresponding NC-RKHS $\mc{H} _{nc} (K^F) =: \scr{H} ^L (F)$, $\mc{H} _{nc} (K^G) =: \scr{H} ^R (G)$ are the left and right free deBranges-Rovnyak spaces of $F,G$.

\subsection{Left vs. Right}

Any element $F \in L^\infty _d$ can be identified with the left free Fourier series:
$$ F \sim F(L) := \sum _{\alpha \in \F ^d}  F_\alpha L^\alpha ; \quad \quad F_\alpha := \ip{L^\alpha 1}{F 1}. $$
That is, $F$ is identified with its \emph{symbol}:
$$ f:= F1 = \sum _{\alpha \in \F ^d}  F_\alpha L^\alpha 1 \in F^2 _d, $$
and we say that $F(L) = M^L _{f}$ acts as left multiplication by $f=F1$. In general the free Fourier series does not converge in SOT or WOT, but the Ces\`aro
sums converge in the strong operator toplogy (SOT) to $F$ \cite{DPac}.

Similarly, in the operator valued setting, any $F \in L^\infty _d \otimes \L (\H , \J)$ is written $F=F(L)= M^L _{f}$, where the symbol, $f \in F^2 _d \otimes \L (\H, \J)$ is defined by 
$$ f:= F (1 \otimes I_\H) = \sum _\alpha L^\alpha 1 \otimes F_\alpha ; \quad \quad F_\alpha \in \L (\H, \J). $$ In this case the operator-valued free holomorphic function $F(Z)$ takes values in $(\C ) _{nc} \otimes \L (\H , \J)$. 

We can also identify any $G \in R^\infty _d$ with its symbol:
$$ g := G1 = \sum _{\alpha \in \F ^d} G_\alpha L^\alpha 1, $$ then we can view $G$ as right multiplication by $g(Z)$,
$$ G = M^R _{g(Z)}. $$ Alternatively, we can write
$$ g = \sum _{\alpha \in \F ^d} G_{\alpha ^\dag} R^\alpha 1, $$ so that
$$ G = M^R _{g(Z)} = g^\dag (R), \quad  \mbox{where} \quad g^\dag (Z) := \sum _{\alpha \in \F ^d} G_{\alpha ^\dag} Z^\alpha. $$
That is, if $G \in R^\infty _d$ acts as right multiplication by the free NC holomorphic function $G (Z)$, then
$ M^R _{G(Z)} $ is identified with the right free Fourier series $G^\dag (R)$ (whose Ces\`aro sums converge SOT to $G$).

\begin{remark} \label{rightproduct} In the right operator-valued setting, suppose that $G (R) := g  (R) \otimes X \in R^\infty _d \otimes \L (\H , \J)$ and $F :=f  (R) \otimes Y \in R^\infty _d \otimes \L (\J , \K )$ with $f,g \in R^\infty _d$, $X \in \L (\H, \J)$, and $Y \in \L (\J , \K)$. If $H = FG$, then observe that 
$$ H ^\dag (Z) = g ^\dag (Z) f ^\dag (Z) \otimes YX. $$ This extends to a `right product' for arbitrary operator-valued free holomorphic functions on $\B ^d _\N$,  $H (Z) = F(Z) \bullet _R G(Z)$. In the scalar-valued setting this simply reduces to $F(Z) \bullet _R G(Z) = G(Z) F(Z)$. 
\end{remark}

\subsection{Operator-valued free multipliers}
\label{freedBRspace}

It will also be convenient to consider operator-valued (left and right) free multipliers between vector-valued free Hardy spaces. Namely, if $\H$ is an auxiliary Hilbert space, one can consider the NC-RKHS $H^2 (\B ^d _N ) \otimes \H$ of $\H$-valued NC functions on $\B ^d _\N$. This NC-RKHS has the operator-valued CPNC kernel:
$$ K(Z,W) \otimes I_\H, $$ and is spanned by the elements 
$$ \kz h := \kz \otimes h, \quad \quad h \in \H,$$ with inner product defined by 
$$ \ip{\kz h}{\kw g}:= \ipcn{y}{K(Z,W)[vu^*]x} \cdot \ip{h}{g}_\H, $$ for $h,g \in \H$, $Z \in \B ^d _n, W \in \B ^d _m$, $v , y \in \C ^n$ and $u, x \in \C ^m$. 
We will write $H^\infty _L (\B ^d _\N) \otimes \L (\H , \J)$ in place of $\mr{Mult} ^L (H^2 (\B ^d _\N ) \otimes \H , H^2 (\B ^d _\N ) \otimes \J )$, the $WOT-$closed \emph{left multiplier space} between these vector-valued free Hardy spaces. (That is, we write $H^\infty _L (\B ^d _\N) \otimes \L (\H , \J)$ in place of the weak operator topology closure of this algebraic tensor product). The operator-valued Schur classes, $\scr{L} _d (\H , \J), \scr{R} _d (\H ,\J)$ are then the closed unit balls of these operator-valued left and right (\emph{resp}.) multiplier spaces.

If $A \in \scr{L} _d (\H , \J)$, consider the operator-valued CPNC kernel:
$$ K ^{A} (Z, W) [\cdot]= K(Z,W)[\cdot] \otimes I _\J - A(Z) (K(Z,W)[\cdot] \otimes I_\H) A(W) ^*. $$ Here, $K$ is the free Szeg\"{o} kernel. This is the left free deBranges-Rovnyak CPNC kernel of $A$, and the corresponding NC-RKHS, $\H _{nc} (K^A) =: \scr{H} ^L (A)$ is called the left free deBranges-Rovnyak space of $A$.
Namely, $\scr{H} ^L (A)$ is the closed linear span of vectors of the form
$K^A \{ Z, y, v \} g$ whose inner product is defined by:
\ba && \ip{K^A \{ Z,y,v \} g }{K^A \{ W,x,u \} f} _{\scr{H} ^L (A)}  \nn \\
&:= & \left( y \otimes g , \left(K(Z,W)[v u^*] \otimes I_\J - A(Z) (K(Z,W) [v u^*] \otimes I_\H ) A(W)^* \right) x \otimes f \right)_{\C ^n \otimes \J}. \nn \ea
In this vector-valued setting, for $Z \in \B ^d _n$ we write $K^A _Z : \C ^{n\times n} \otimes \J \rightarrow \scr{H} ^L (A)$ for the kernel map $K^A _Z (yv^* \otimes h) = K^A \{ Z,y,v \} h$. 

On the other hand, if $B \in \scr{R} _d (\H, \J)$, then the right free deBranges-Ronvyak space $\scr{H} ^R (B)$ is spanned by the vectors 
$K^B \{ Z , v , y \} g$ with inner product: 
\ba & &  \ip{K^B \{ Z,y,v \} g }{K^B \{ W,x,u \} f} _{\scr{H} ^R (B)} \nn \\
&:=& \left( y \otimes g, K^B (Z,W) [v u^* \otimes I_\H ]   x \otimes f \right)_{\C ^n \otimes \J} \nn \\
K^B (Z,W) & = &K(Z,W) \otimes I _\J  
- (K(Z,W) \otimes I _\J) [B^\dag (Z) ( \cdot \otimes I_\H ) B^\dag (W)^* ]. \nn \ea 
It is not difficult to see that free operator-valued holomorphic functions $A,B$ on $\B ^d _\N$ belong to the left or right free Schur classes if and only if the above NC deBranges-Rovnyak kernels are (completely) positive.

\begin{remark}{ (Right Product)}
    If $F \in \mr{Mult} ^R (K_1 ,K_2) \otimes \L (\H , \J)$ is a right operator-valued multiplier between vector-valued NC-RKHS on (say) the open unit NC ball $\B ^d _\N$, then one can easily verify that for any $g \in \J$,
    $$ (M^R _F) ^* K_2 \zyv g= K_1 \{ Z , y , F(Z) \bullet _R v  \} g, $$ where for any $f \in \H _{nc} (K_1)$ we define:
$$ \ip{K_1 \{ Z , y , F(Z) \bullet _R v \} g}{f}  =  \ip{y \otimes h}{ F(Z) \bullet _R f(Z) v}. $$
Also, in the above, given any element $f$ of a $\H-$valued NC-RKHS $\H _{nc} (K)$, note that $f(Z) \in \C ^{n\times n} \otimes \H$, so that $f(Z) v$ is to be interpreted as an element of $\C ^n \otimes \H$. 
\end{remark}

\subsection{Coefficient evaluation and free formal RKHS}

Let $K$ be an operator-valued CPNC kernel on $\B ^d _\N$ whose Taylor-Taylor series about $0 \in \B ^d _1 = \B ^d$ converges absolutely on $\B ^d _\N$, and uniformly on compacta:
$$ K(Z,W) [P] = \sum _{\alpha , \beta \in \F ^d} K_{\alpha , \beta} Z^\alpha [P] (W^*) ^{\beta ^\dag}; \quad \quad K_{\alpha , \beta} \in \L (\H ). $$ Any right or left (operator-valued) deBranges-Rovnyak kernel has this property, for example. The \emph{coefficient kernel} function $K _{(\cdot , \cdot) } : \F ^d \times \F ^d \rightarrow \L (\H)$ is then an operator-valued \emph{free formal kernel} in the sense of \cite{Ball2003rkhs,BMV} (see also \cite{JMfree} which develops free Aleksandrov-Clark theory using the free formal RKHS setup).

If $F(Z):= \sum _\alpha F_\alpha Z^\alpha \in \H _{nc} (K)$, then for any $\alpha \in \F ^d$, the linear $\H-$valued map defined by coefficient evaluation:
$$ K _\alpha ^* (F) := F_\alpha \in  \H, $$ is bounded.  The Hilbert space adjoint $K_\alpha := (K _\alpha ^* ) ^* : \H \rightarrow \H _{nc} (K)$ will be called the coefficient kernel map, and one always has 
$$ K _{\alpha , \beta} = K_\alpha ^* K_\beta \in \L (\H ), $$ and
$$ K_\beta (Z) = \sum _{\alpha \in \F ^d} K_{\alpha ,\beta} Z^\alpha. $$ Observe that $K_{\alpha,\beta}$ is a positive kernel function in the classical sense on the discrete set $\F ^d$.

\subsection{The free Herglotz-Schur classes}

\begin{defn}
 The left free Herglotz-Schur class, $\scr{L} _d ^+$, is the set of all accretive matrix-valued free NC functions on $\B ^d _\N$.
\end{defn}

If $H \in \scr{L} _d ^+$, then $H(Z)$ is an accretive matrix (positive semi-definite real part), and it follows that
$$ B_H (Z) : =  (H(Z) -I)(H(Z) +I ) ^{-1} , $$ is a contractive free function on $\B ^d _\N$, \emph{i.e.}, $B \in \scr{L} _d$. Conversely, given such a $B$,
$$ H_B (Z) := (I+ B(Z) ) ( I - B(Z) ) ^{-1}, $$ has non-negative real part. That is, this fractional linear transformation, the \emph{Cayley Transform}, defines a bijection between $\scr{L} _d$ and $\scr{L} _d ^+$.
The right free Herglotz-Schur class, $\scr{R} _d ^+$ is the image of $\scr{L} _d ^+$ under the transpose map, $\dag$, and we similarly define the operator-valued free Herglotz-Schur classes $\scr{L} _d ^+ (\H) = \scr{L} _d ^+ (\H , \H )$ as the image of $\scr{L} _d (\H)$ under Cayley transform.      

\begin{remark} \label{HergSmir}
By \cite[Lemma 3.2, Lemma 3.3]{JM-freeSmirnov}, if $B \in \scr{R} _d$ or $\scr{L} _d$, then $1-B$ is outer (and necessarily invertible on $\B ^d _\N$). Moreover, the free Herglotz-Schur classes are contained in the free Smirnov classes, which can be identified with closed, densely-defined (generally unbounded) right and left multipliers of $H^2 (\B ^d _\N )$ \cite{JM-freeSmirnov}. In particular, by \cite[Corollary 3.13, Corollary 3.15]{JM-freeSmirnov}, if $B \in \scr{L} _d$ or $A \in \scr{R} _d$, then the free polynomials belong to the domains of and are cores for both $H_B (L) ^*$ and $H_A (R) ^*$.
\end{remark}

\section{Free Aleksandrov-Clark measures}

Let $\A _d := \left( \bigvee _{\alpha \in \F ^d} L^\alpha \right) ^{-\| \cdot \|}$, the free disk algebra.
In the case where $d=1$, $L=S$ (the shift), and we recover the classical disk algebra $\A _1 = \A (\D )$ of bounded analytic functions on $\D$ with continous extensions to the unit circle $\T=\partial \D$.

Recall the Herglotz representation formula for Herglotz functions (analytic functions with non-negative real part) on the disk: If $H$ is a Herglotz function, then there is a unique finite positive Borel measure $\mu$ on $\T$ so that:
$$ H(z) = i \im{H(0)} + \int _\T \frac{1+z\ov{\zeta}}{1-z\ov{\zeta}} \mu (d\zeta ). $$ 
As discussed above, any such $H \in \scr{S} ^+ := \scr{L} _1 ^+$ is the Cayley transform of some contractive analytic Schur class function $b \in \scr{S} := [H^\infty (\D ) ]_1 = \scr{L} _1$, $H = H_b = \frac{1+b}{1-b},$ and the measure $\mu =: \mu _b$ is called the Aleksandrov-Clark measure of $b$. This defines bijections (modulo imaginary constants) between $\scr{S}, \scr{S} ^+$, and the set of all finite positive Borel measures on $\T$. 

Any finite positive Borel measure $\mu$ on $\T$ can be identified with a positive linear functional on the disk algebra operator system: $\A _1 + \A _1 ^*$: 
$$ \hat{\mu} (S^n) := \int _\T \zeta ^n \mu (d\zeta ). $$ 
The theory of closed, densely-defined operators affiliated to the shift \cite{Suarez,Sarason-ub}, implies that if $b \in \scr{S}$, then multiplication by $H_b \in \scr{S} ^+$ is a closed, densely-defined (and accretive) operator on $H^2 (\D)$, and that $\bigvee S^n 1$ is a core for $M_{H_b} ^*$. It is then easy to verify the following formula for $\hat{\mu }_b$:
\be \hat{\mu _b} (S^n )  =  \frac{1}{2} \left(\ip{H_b (S) ^* 1}{S^n 1} _{H^2} + \ip{1}{H_b(S)^* S^n 1} _{H^2} \right), \label{Clarkform} \ee where $m$ denotes normalized Lebesgue measure on $\T$. Equivalently, using that any kernel vector $k_z$, for $z \in \D$, is necessarily an eigenvector for $H_b (S) ^*$ with eigenvalue $H_b (z) ^*$, one can check that
\be H_b (z)   =  i\im{H_b (0)} + \hat{\mu} _b \left( (I+zS^*)(I-zS^*) ^{-1} \right). \ee
This Clark functional formula (\ref{Clarkform}) extends verbatim to the non-commutative several-variable setting: 
\begin{defn}
Let $(A,B) \in \scr{R} _d \times \scr{L} _d$ be a transpose-conjugate ($A = B^\dag$) free Schur class pair. The \emph{Clark functional} of $(A , B)$ is the self-adjoint linear functional: $\mu _A = \mu _B : \bigvee L^\alpha 1 + \bigvee (L^\alpha ) ^* \rightarrow \C$
defined by:
\ba  \mu _B (L^\alpha ) & := & \frac{1}{2} \left( \ip{H_A (R) ^* 1 }{L^\alpha 1 } _{F^2 _d} + \ip{1}{H_A (R) ^* L^\alpha 1 } _{F^2 _d} \right) \nn \\
& = & \frac{1}{2} \left( H_B (0) \delta _{\alpha , \emptyset} +  \ip{1}{H_A (R) ^* L^\alpha 1 } _{F^2 _d} \right). \nn \ea
\end{defn}
In the above, recall that the free monomials always belong to the domain of $H_A (R) ^*$, see Remark \ref{HergSmir}. To simplify notation, we will write $\A _d + \A _d ^*$ in place of its norm closure: $\A _d + \A _d ^* = (\A _d + \A _d ^* ) ^{-\| \cdot \| }$.

\begin{lemma}
$\mu _B$ extends by continuity to a positive, bounded linear functional on the norm-closed operator system $\A _d + \A _d ^*$.
\end{lemma}

\begin{proof}
That $L^\alpha 1$ belongs to the domain of $H_A (R) ^*$ (and that, in fact, free polynomials are a core for $H_A (R) ^*$) follows from \cite[Corollary 3.9, Corollary 3.10, Remark 3.12]{JM-freeSmirnov}. It is easy to check that $\mu _B$ is positive on $\bigvee _\alpha L^\alpha  + \bigvee _\alpha (L^\alpha ) ^*$. Since this is a positive linear functional, its norm is given by
$$ \| \mu _B \| = \mu _B (I) = \re{ (H_B ) _\emptyset } = \re{H _B (0)} < \infty. $$ Hence, $\mu _B$ extends by continuity to a bounded positive linear functional on the free disk operator system.
\end{proof}

\begin{thm} \label{Clarkfunthm}
The map $B \mapsto \mu _B$ is a bijection from $\scr{L} _d$ onto the set of positive linear functionals on the free disk system,
and one has the \emph{free Herglotz formula}: For any $Z \in \B ^d _n$,
$$ H_B (Z) = i \im{H_B (0_n)} + (\id _n \otimes \mu _B ) \left( (I_{n \times F^2} + ZL^* ) (I_{n \times F^2} -ZL^* ) ^{-1}) \right). $$
\end{thm}
In the above,
$$ ZL^* := (Z \otimes I_{F^2} ) (I_n \otimes L) ^* = Z_1 \otimes L_1 ^* + ... + Z_d \otimes L_d ^*, $$ and
$$ I_{n \times F^2 } := I_n \otimes I_{F^2 _d}.$$ Also note that:
$$ (I -ZL^* ) ^{-1} = \sum _{k=0} ^\infty (ZL^* ) ^k = \sum _{\alpha \in \F ^d }  Z^\alpha \otimes (L^*) ^\alpha. $$

\begin{proof}
Let $H_B (Z) = \sum _\alpha H_\alpha Z^\alpha$ be the Taylor-Taylor series of $H_B (Z)$ about $Z =0 \in \B^d = \B ^d _1$. Consider $(\id _n \otimes \mu _B) \left( (I -ZL^* ) ^{-1} \right)$:
\ba (\id _n \otimes \mu _B ) \left( (I - ZL^* ) ^{-1} \right) & = & \sum _{\alpha} Z^\alpha \mu _B ( (L^*) ^\alpha ) \nn \\
& = & \sum _\alpha Z^\alpha \frac{1}{2} \left( \ip{H_A (R) ^* L^{\alpha ^\dag} 1}{1} +  \delta _{\alpha ^\dag , \emptyset} H  _\emptyset ^* \right) \nn \\
& = & \frac{1}{2} I_n H  _\emptyset ^* + \frac{1}{2} \sum _\alpha Z^\alpha \sum _{\ga \beta = \alpha ^\dag} \ip{ L^\ga }{1} H _{\beta ^\dag}  \nn \\
& = & \frac{1}{2} H_B (0_n ) ^* + \frac{1}{2} \sum _\alpha Z^\alpha H _\alpha \nn \\
& = & \frac{1}{2} H_B (0 _n ) ^* + \frac{1}{2} H_B (Z). \nn \ea
In the above, note that $H_A (R) = U_\dag H_B (L) U_\dag ^*$, so that $H_A (R) =M^R _{H_B ^\dag (Z)}$. We also used the fact that if $$ F(Z) = \sum _\alpha Z^\alpha F_\alpha, $$ is any free holomorphic function so that $M^R _{F(Z)}$ is densely-defined, then the monomials $L^\alpha 1 \in \dom{(M^R _{F(Z)}) ^*}$, and 
$$ (M^R _{F(Z)}) ^* L^\alpha 1 = \sum _{\ga \beta = \alpha} L^\ga 1 F_\beta ^*, $$ see for example \cite[Corollary 3.13]{JM-freeSmirnov} and \cite[Lemma 2.3]{JMfree}.  
Using that $(I +ZL^* ) (I -ZL^* ) ^{-1} = 2 (I-ZL^* ) ^{-1} - I$, we obtain:
$$ (\id _n \otimes \mu _B ) \left( (I +ZL^* ) (I -ZL^* ) ^{-1} \right) = H_B (0_n) ^* + H_B (Z) - \re{H_B (0 _n) }, $$
and the formula follows.

Conversely, starting with a positive linear functional on the free disk system, this Herglotz formula defines a free Herglotz function, and by Cayley transform
we obtain a free Schur function whose Clark functional is the original functional. This shows $B \mapsto \mu _B $ is surjective.
\end{proof}

\begin{remark}
Replacing $F^2 _d$ with $F^2 _d \otimes \H$ where $\H$ is a separable or finite-dimensional Hilbert space, the above results are easily extended to the operator-valued setting of $B \in \scr{L} _d (\H)$.  In this operator-valued setting we define the Clark map $\mu _B : \A _d + \A _d ^* \rightarrow \L (\H)$ by the formula:
$$\mu _B (L^\alpha ) := \frac{1}{2} H_\emptyset \delta _{\alpha , \emptyset} I_\H + \frac{1}{2} \left( \ip{1}{\cdot {1}} \otimes \mr{id} _\H \right) \left( H_A (R) ^* (L^\alpha \otimes I_\H) \right). $$ Again this defines a bijection between the operator-valued free Schur classes, the operator-valued free Herglotz-Schur classes, and $\L (\H)-$valued completely positive maps on the free disk system.
\end{remark}

\begin{remark}
Theorem \ref{Clarkfunthm} was first obtained by Popescu in \cite[Section 5, Theorem 5.3]{Pop-freehar} (with a different, but equivalent formula for the Clark functional). Given $B \in \scr{L} _d$, the Clark functional $\mu _B$ can also be defined in terms of the Fourier series coefficients of the free Herglotz function $H_B$, as in \cite[Section 4]{JMfree}.
\end{remark}

\section{Free Cauchy transforms}

In his seminal paper on unitary perturbations of the shift (see \cite{Sarason-dB} for the fully general, non-inner case), D.N. Clark showed that there is a canonical isometry, the \emph{weighted Cauchy Transform}, $\mc{F} _b$, from $H^2 (\mu _b)$, the closure of the analytic polynomials in $L^2 (\mu _b)$ (the Hilbert space of functions on $\T$ which are square-integrable with respect to $\mu _b$), onto $\scr{H} (b)$, the deBranges-Rovnyak space of $b \in \scr{S} =\scr{L} _1$ \cite{Clark1972}: For any polynomial, $p \in H^2 (\mu _b)$, 
$$ \left( \mc{F} _b  p \right)  (z) := (I -b(z) ) \int _\T  \frac{1}{1-z\zeta ^*}  p(\zeta) \mu _b (d\zeta). $$ 
One can also define an unweighted Cauchy Transform, $\mc{C} _b$, from $H^2 (\mu _b)$ onto $\scr{H} _+ (H_b) := \H (K^{H_b} )$, the \emph{Herglotz space} of $b$, the unique RKHS corresponding to the positive sesqui-analytic Herglotz kernel:
$$ K^{H_b} (z,w) := \frac{1}{2} \frac{H_b (z) + H_b (w) ^*}{1-zw^*}; \quad \quad z,w \in \D. $$ 
With a bit of algebra, one can verify that 
$$ K^{H_b} (z,w) = (I-b(z) ) ^{-1} k^b (z,w) (I-b(w) ^* ) ^{-1}, $$ where $k^b$ is the deBranges-Rovnyak kernel of $b$. The theory of RKHS then implies that the multiplier
$$ \mc{U} _b  := M _{(I-b)} : \scr{H} _+ (H_b) \rightarrow \scr{H} (b), $$ is an isometry of the Herglotz space onto the deBranges-Rovnyak space of $b$. The Cauchy Transform $\mc{C} _b : H^2 (\mu _b) \rightarrow \scr{H} _+ (H_b)$ is the linear map defined by:
$$ \left( \mc{C} _b (p) \right) (z) := \int _\T  \frac{1}{1-z\zeta ^*} p(\zeta) \mu _b (d\zeta), $$ and this extends to an isometry of $H^2 (\mu _b)$ onto the Herglotz space of $b$ so that $\mc{F} _b = \mc{U} _b \mc{C} _b$.

In \cite{JMfree}, we extended the notions of Cauchy Transform and weighted Cauchy Transform to the (operator-valued and) free setting using the theory of free formal RKHS. Here we describe Cauchy transforms in the setting of NC-RKHS: Assume that  $B \in \scr{L} _d (\H)$  or that $A \in \scr{R} _d (\H )$ are in the left or right operator-valued free Schur classes and that $A = B^\dag$ so that $\mu _A = \mu _B$. The \emph{free left Herglotz space} $\scr{H} ^L _+ (H_B) = \H _{nc }  (K ^L)$ is the NC-RKHS corresponding to the free left Herglotz kernel:
$$ K^L (Z,W) [P] := \frac{1}{2} \left(H_B (Z) (K(Z,W) [P] \otimes I_\H) + (K(Z,W) [P] \otimes I _\H) H_B (W) ^* \right), $$ where $K$ is the free Szeg\"{o} kernel. As in the classical theory, it is straightforward to verify that 
$$ \mc{U} _B := M^L _{(I-B)} : \scr{H} ^L _+ (H_B) \rightarrow \scr{H} ^L (B), $$ is an onto isometric left free multiplier. If $A \in \scr{R} _d (\H)$, then the right free Herglotz space, $\scr{H} ^R _+ (H_A)$ is defined similarly, and $\mc{U} _A = M^R _{(I -A^\dag (Z) )}$ is an isometric multiplier of $\scr{H} ^R _+ (H_A)$ onto $\scr{H} ^R (A)$.

We can expand this kernel in a formal power series (actually a convergent Taylor-Taylor series about $0$):
$$ K^L (Z,W)[P] := \sum _{\alpha ,\beta } K^L _{\alpha ,\beta} Z^\alpha P (W^* ) ^{\beta ^\dag}, $$ where
$$ K^L _{\alpha ,\beta } := \mu _B ((L^{\alpha ^\dag }) ^* L^{\beta ^\dag } ), $$ is the free left formal Herglotz kernel defined in \cite[Proposition 4.5]{JMfree}. In the right case, if $A \in \scr{R} _d (\H)$ one simply defines
$$ K^R _{\alpha ,\beta} := \mu _A ( (L ^\alpha ) ^* L^\beta ). $$ 

As described in \cite{JMfree}, given a transpose-conjugate pair $(A,B) \in \scr{R} _d (\H) \times \scr{L} _d (\H )$, the appropriate generalization of the (analytic part of the) Clark measure space $H^2 (\mu _b)$ is the Stinespring-Gelfand-Naimark-Segal (S-GNS) space or \emph{free Hardy space} of $\mu _B : \A _d + \A _d ^* \rightarrow \L (\H )$, $F^2 (\mu _B)$. This is defined as the Hilbert space completion of $\A _d \otimes \H$ (modulo vectors of zero length) with respect to the pre-inner product:
$$ \ip{L^\alpha \otimes h}{L^\beta \otimes g}_{\mu _B} := \ip{h}{\mu _B \left( (L^\alpha ) ^* L^\beta \right) g}_\H. $$ The semi-Dirichlet property: $(\A _d ) ^* \A _d = \A _d + \A _d ^*$ (norm closure) ensures this is a well-defined inner product, and the left regular representation: $L ^\alpha \mapsto \pi _B (L ) ^\alpha$ where 
$$ \pi _B (L^\alpha ) p(L) \otimes h := L^\alpha p(L) \otimes h, $$is completely isometric, unital, and extends to a $*$-representation of the Cuntz-Toeplitz $C^*-$algebra. We will set $\Pi ^B _k := \pi _B (L _k)$, so that $\Pi ^B = \pi _B (L)$ is the S-GNS row isometry on $F^2 (\mu _B)$.  This also provides a S-GNS formula for $\mu _B$:
$$ \mu _B (L^\alpha) = [I \otimes] _B ^* \pi _B (L) ^\alpha [I \otimes] _B, $$ where $[I\otimes ] _B : \H \rightarrow F^2 (\mu _B)$ is the bounded embedding:
$$ [I\otimes] _B h := I \otimes h. $$

The left and right Cauchy transforms, $\mc{C} ^L : F^2 (\mu _B ) \rightarrow \scr{H} ^L _+ (H_B)$ and $\mc{C} ^R : F^2  (\mu _A) \rightarrow \scr{H} ^R _+ (H_A)$ are then defined by
\be \mc{C} ^L (L^\alpha \otimes h) := K^L _{\alpha ^\dag}h, \quad \mbox{and}, \quad \mc{C} ^R (L^\alpha \otimes h ) := K^R _\alpha h. \label{freeCT} \ee 

Observe that if $H_B (Z) = \sum _\alpha H_\alpha Z^\alpha$, then $H_A (Z) = \sum _\alpha H_{\alpha ^\dag} Z^\alpha$. One can then calculate that:
$$ K^L _{\alpha, \beta} = \left\{ \begin{array}{cc} \frac{1}{2} H  _{(\beta ^\dag \sm \alpha ^\dag) ^\dag } ^* & \beta ^\dag > \alpha ^\dag \\ 
\frac{1}{2}  H _{(\alpha ^\dag \sm \beta ^\dag) ^\dag } & \alpha ^\dag > \beta ^\dag \\
\re{H _\emptyset} & \alpha = \beta  \\
0 & \mbox{else}. \end{array} \right. $$
In the above, we write $\beta \geq \alpha$ if $\beta = \alpha \ga$, and $\beta > \alpha$ if $\beta \geq \alpha$ and $\beta \neq \alpha$.
Similarly,
$$ K^R _{\alpha, \beta} = \left\{ \begin{array}{cc} \frac{1}{2} H  _{(\beta  \sm \alpha) ^\dag } ^* & \beta  > \alpha  \\ 
\frac{1}{2}  H _{(\alpha  \sm \beta) ^\dag } & \alpha  > \beta  \\
\re{H _\emptyset} & \alpha = \beta  \\
0 & \mbox{else}. \end{array} \right. $$

These formulas follow easily from the Clark map formula. For example if $\beta > \alpha$ (and $\H = \C$) then 
\ba K^R _{\alpha ,\beta} & = & \mu _A ( L^{\beta \sm \alpha} ) \nn \\
& = & \frac{1}{2} \left( \ip{H_A (R) ^* 1}{L^{\beta \sm \alpha} 1} + \ip{1}{H_A (R) ^* L^{\beta \sm \alpha} 1 } \right) \nn \\
& = & 0 + \frac{1}{2} \sum _{\ga \la = \beta \sm \alpha} \ip{ 1}{L^{\ga  } 1 } H_{\la ^\dag} ^*  \nn \\
& = &  \frac{1}{2} (H) _{(\beta \sm \alpha)^\dag} ^*. \nn \ea
The above formulas allow one to alternatively define the Clark map of $B$ in terms of the Fourier series coefficients of the Herglotz functions $H_B$, as was done in \cite[Section 4]{JMfree}.

\begin{lemma}
The left free Cauchy transform acts as:
$$ (\mc{C} ^L p(L) 1) (Z) = (\id _n \otimes \mu _B ) \left( (I_{n\times F^2} -ZL^*) ^{-1} (I_n \otimes p(L)) \right). $$
\end{lemma}

\begin{proof}
It suffices to check on monomials, so take $p(L) = L^{\beta ^\dag} 1$. Then, the above becomes:
$$ \sum _\alpha Z^\alpha \mu _B ( (L^{\alpha ^\dag} ) ^* L^{\beta ^\dag} ) = K ^L _\alpha (Z), $$ as claimed.
\end{proof}

\begin{lemma} \label{CTofkernel}
Given any $Z \in \B ^d _n$ and $ v , y \in \C ^n$,
\ba K ^L  \{ Z, y ,v  \} &= &\mc{C} ^L \sum _{\alpha} \ipcn{Z^\alpha v }{y} L^{\alpha ^\dag} 1 \nn \\
& = & \mc{C} ^L \ipcn{v \otimes I_{F^2 _d}}{(I-Z^*L)^{-1} (y \otimes 1)}. \nn \ea
\end{lemma}

\begin{proof}
For any $F \in \scr{H} ^{L,+} (H_B)$, we have that
\ba \ip{K ^L \{ Z,y,v \} }{F} &=& \ipcn{y}{F(Z) v} \nn \\
& = & \sum _\alpha \ipcn{y}{Z^\alpha v} F_\alpha \nn \\
& = & \sum _\alpha \ipcn{y}{Z^\alpha v} \ip{K ^L _\alpha}{F} \nn \\
& = & \ip{\sum _\alpha \ipcn{Z^\alpha v}{y} K ^L  _\alpha}{F}. \nn \ea
The above proves that $K ^L \zyv = \sum _\alpha \ipcn{Z^\alpha v}{y} K ^L _\alpha$, and by definition, $K^L _\alpha = \mc{C} ^L L^{\alpha ^\dag} 1$. Hence we have that:
\ba K^L \{ Z , y ,v \} & = & \mc{C} ^L \sum \ipcn{Z^\alpha v}{y} L^{\alpha ^\dag} 1 \nn \\
& = & \mc{C} ^L \sum \ipcn{ v \otimes I _{F^2 _d}}{ (Z^*) ^{\alpha ^\dag} \otimes L^{\alpha ^\dag} (y \otimes 1) } \nn \\
& = &  \mc{C} ^L \ipcn{v \otimes I_{F^2 _d}}{(I-Z^*L)^{-1} (y \otimes 1)}. \nn \ea
\end{proof}

\begin{remark}
We also have the formula:
\ba K ^L \wxu (Z) & = & \sum _{\alpha} \ipcn{W^\alpha u}{x} (\id _n \otimes \mu _B ) \left( (I -ZL^*) ^{-1} (I_n \otimes L^{\alpha ^\dag} ) \right) \nn \\
& = & (\ipcn{\cdot u}{x} \otimes \mu _B ) \left( (I - ZL^* ) ^{-1} \sum _\alpha (W^*) ^{\alpha ^\dag} \otimes L^{\alpha ^\dag} \right) \nn \\
& = &  (\ipcn{\cdot u}{x} \otimes \mu _B ) \left( (I - ZL^* ) ^{-1} (I - LW^* ) ^{-1} \right). \nn \ea
The above is the free version of the commutative Cauchy transform formula,
$$ \mu _b ((I-zL^*)^{-1} (I - Lw^* ) ^{-1} ) = K^b (z,w), $$ from \cite[Proposition 2.6, Subsection 2.8]{JM}. Here $K^b (z,w)$ is the positive Herglotz kernel for $b$ in the Schur class of contractive Drury-Arveson space multipliers.
\end{remark}

If $A = B^\dag$ so that $\mu _A = \mu _B$ and $F^2 (\mu _A ) = F^2 (\mu _B)$, then the weighted free Cauchy transforms $\mc{F} ^L , \mc{F} ^R$, are defined as:
$$ \mc{F} ^L = M^L _{(I-B)} \mc{C} ^L, \quad \mbox{and}, \quad \mc{F} ^R = M^R _{(I-A^\dag )} \mc{C} ^R, $$
and these are isometries of $F^2 (\mu _B)$ onto $\scr{H} ^L (B)$ and $\scr{H} ^R (A)$, respectively.

\subsection{Cauchy Transform of the Stinespring-GNS representation} \label{CTofSGNS}

As in the commutative setting, if $B \in \scr{L} _d (\H)$ we define
$$ V^B := \mc{C} ^L \pi _{\mu _B} (L) (\mc{C} ^L ) ^*, $$ a row isometry on the left Herglotz space
$\scr{H} ^L _+ (H_B)$. 

\begin{prop} \label{CTthm}
    The range $\mc{R}$ of the row isometry $V^B$ is:
$$ \mc{R} :=\bigvee \left( K^{H_B} \{ Z , y , v \} - K^{H_B} \{ 0 _n , y, v \} \right) = \bigvee _{\alpha \neq \emptyset} K^{H_B} _\alpha, $$ and
for any $Z \in \B ^d _n, \ v, y \in \C^n$, and $j=1, \dots , d$, 
$$ 
(V_j^B) ^* \left( K^{H_B} \{ Z , y , v \} - K^{H_B} \{ 0 _n , y, v \} \right) = K ^{H_B} \{ Z , y , Z_j v \} 
$$ 
(so that the span of all such vectors is dense in $\scr{H} ^L _+ (H_B) \otimes \C ^d$).

The image of $\ran{V^B}$ under $(\mc{C} ^L) ^*$ is $F^2 _0 (\mu _B) = \bigvee _{\alpha \neq \emptyset} L^\alpha \otimes \H$, the closure of the non-constant free monomials in $F^2 (\mu _B)$. If $F \in \scr{H} ^L _+ (H_B)$ is orthogonal to $\ran{V^B}$, then there is a $f \in \H$ so that for any $Z \in \B ^d _n$, 
$$ F(Z) = I_n \otimes f, $$ \emph{i.e.} $F \equiv f$ is constant-valued. 
\end{prop}

\begin{proof}
By the proof of \cite[Lemma 3.14]{JM-freeSmirnov}, for any $\alpha \in \F ^d$, one can find jointly nilpotent $Z \in \B ^d _n$ and $v , y \in \C ^n$ with $n =|\alpha | +1$ so that 
$$ \ipcn{Z^\beta v}{y} = \delta _{\alpha, \beta}. $$ It then follows from Lemma \ref{CTofkernel} and the definition of left free Cauchy transform that $K ^L \{ Z , v, y \} = K^L _\alpha$. This shows that the two formulas for $\mc{R}$ above are the same. By definition, $\mc{C} ^L (L^\alpha \otimes h) = K^L _{\alpha ^\dag} h$, and it follows that the image of $\mc{R}$ under inverse Cauchy transform is $F^2 _0 (\mu _B) = \ran{\pi _B (L)}$. Since $V^B$ and $\Pi ^B$ are unitarily equivalent under Cauchy transform, it follows that $\mc{R} = \ran{V^B}$. If $F \in \scr{H} ^L _+ (H_B)$ is orthogonal to $\ran{V^B}$, set 
$ f:= (K_0 ^{H_B} ) ^* F \in \H. $ Then, for any $Z \in \B ^d _n$ and $v^*,y \in \C ^n$, 
\ba 0 & = & \ip{K ^{H_B} \zyv}{F} - \ip{K^{H_B} \{0_n , y , v \}}{F} \nn \\
& = & \left( y \otimes I_\H , F(Z) v \right) - (I_n \otimes f) \ipcn{y}{v}, \nn \ea and it follows that $F(Z) = I_n \otimes f$.

The second claim is a straightforward calculation: for each $j=1, \dots , d$,
\ba (V_j^B) ^* \left( K ^{L} \{ Z, y, v \} - K^L \{ 0_n , y, v \} \right) & = & \mc{C} ^L \pi (L_j) ^* \sum _{\alpha \neq \emptyset} \ipcn{Z^\alpha v}{y} L^{\alpha ^\dag} 1 \nn \\
& = & \mc{C} ^L \sum _\beta \ipcn{Z^{\beta j} v}{y} L^{\beta ^\dag} 1 \nn \\
& = & K^L \{ Z , y , Z_jv \}. \nn \ea Since $V^B$ is an isometry, the above shows that 
the closed span of $\oplus_{j=1}^d K^L \{ Z , y , Z_jv \}$ is all of $\scr{H} ^L _+ (H_B) \otimes \C ^d$. 
\end{proof}

\section{Gleason solutions}

\subsection{The free setting}

Fix $A \in \scr{R} _d (\H , \J )$. Exactly as in the commutative setting, we define:

\begin{defn}
    A linear map $X : \scr{H} ^R (A) \rightarrow \scr{H} ^R (A) \otimes \C ^d$ is called a \emph{Gleason solution} for $\scr{H} ^R (A)$ if:
\be Z (X f) (Z) = f(Z) - f(0_n) \quad \quad \forall \ Z \in \B ^d _n. \label{GSHA} \ee
Such an $X$ is \emph{contractive} if 
\be \label{eqn:contractive-gleason-def} X ^* X \leq I - K_0 ^A (K_0 ^A ) ^*, \ee and \emph{extremal} if equality holds.

Similarly, a linear map $\mbf{A} : \H \rightarrow \scr{H} ^R (A) \otimes \C ^d$ is called a \emph{Gleason solution} for $A$ if
\be Z \mbf{A} (Z) = A ^\dag (Z) - A(0) I_n \quad \quad \forall \ Z \in \B ^d _\N. \ee $\mbf{A}$ is \emph{contractive} if 
\be \mbf{A} ^* \mbf{A}  \leq I _\H - A (0) ^* A (0), \ee and \emph{extremal} if equality holds.
\end{defn}

Define: \be \check{X} := L^* \otimes I_\J | _{\scr{H} ^R (A)}, \quad \mbox{and} \quad \check{\mbf{A}} := (L^* \otimes I_\J) A. \ee The right free deBranges-Rovnyak space, $\scr{H} ^R (A)$, is left shift co-invariant, $L_j ^* A \in \scr{H} ^R (A)$, and $\check{X}, \check{\mbf{A}}$ obey the contractivity conditions of Gleason solutions for $\scr{H} ^R (A), A$, respectively \cite[Proposition 4.2]{Ball-Fock}. It is also easy to check that $\check{X}, \check{\mbf{A}}$ are Gleason solutions. For example, given $f = \sum _\alpha f_\alpha L^\alpha 1$, it is clear that 
$$ (\check{X} _j f) (Z) =  \sum _{\alpha } f_\alpha L^{\alpha \sm j} 1, $$ where we set $L^{\alpha \sm j} = L^\beta $ if $\beta = j\alpha$, and $=0$ else. It follows that 
$$ Z (\check{X} f ) (Z) = \sum _{\alpha \neq \emptyset} f_\alpha Z^\alpha = f(Z) - f(0 _n ). $$
Also note that the defining formula (\ref{GSHA}) for a Gleason solution for $\scr{H} ^R (A)$ is equivalent to: $$  (L^* K_Z ^A) ^* (Xf) = (K_Z ^A) ^* f - (K_{0 _n } ^A) ^* f, $$ which can be re-arranged to:
\be (I - X^* L^* ) K_Z ^A = K_{0_n} ^A. \label{kerform} \ee

Given $A \in \scr{R} _d (\H , \J)$, the \emph{support} of $A$ is defined to be 
\be \mr{supp} (A) := \bigvee _{Z \in \B ^d _n; \ v^* , y \in \C ^n}  \left( y \otimes I_\H , A ^\dag (Z) ^*  v^* \otimes \J \right) \subseteq \H. \label{support} \ee

\begin{prop} \label{GSforHBandB}
    Suppose that $A \in \scr{R} _d (\H , \J )$. A linear map $X : \scr{H} ^R (A) \rightarrow \scr{H} ^R (A) \otimes \C ^d$ is a contractive Gleason solution for $\scr{H} ^R (A)$ if and only if there is a contractive Gleason solution $\mbf{A} : \H \rightarrow \scr{H} ^R (A) \otimes \C ^d$ for $A$ so that,
$$ X K ^A \{ W , x , u \}  = K ^A \{ W , W^* x , u \}  -  \mbf{A} \ipcm{u}{A ^\dag (W)^* x} \ \in \L( \J , \scr{H} ^R (A) \otimes \C ^d). $$ $X$ is extremal if $\mbf{A}$ is extremal. Conversely $\mbf{A}$ is extremal if $X$ is extremal and $\mr{supp} (A) = \H$. 
\end{prop}

This is a free analogue of \cite[Theorem 4.4]{JM}. Since the proof is (formally) analogous, we prove only the sufficiency.
\begin{remark}
The expression $\ipcm{u}{A^\dag (W)^* x}$ is to be interpreted as taking values in $\L (\J, \H)$ since, for $W \in \B ^d _m$, $A^\dag (W) ^* \in \C ^{m\times m} \otimes \L (\J , \H)$. Namely, for any $g \in \J$, the above formula can be written:
$$ X K^A \{ W, x, u \} g = K^A \{ W , W^* x , u \} g - \mbf{A} \left( u \otimes I_\H , A^\dag (W)^* x \otimes g \right). $$ 
\end{remark}
\begin{proof}
Let $\mbf{A}$ be a contractive Gleason solution for $A$. We wish to show that the formula in the proposition statement defines a contractive Gleason solution for $\scr{H} ^R (A)$.
To prove this, it is sufficient to check that Formula (\ref{GSHA}) holds on kernel vectors. Namely, it suffices to show that 
\ba & &  \ip{K ^A \{Z,Z^*y,v\} }{X K \{ W , x, u \} }  =  \ipcn{y}{Z (XK^A \{W, x, u \} ) (Z) v}  \nn \\
& = & \ipcn{y}{ \left(K^A (Z,W)[v u^*] - K^A (0_n , W) [vu^*] \right) x  } \in \L (\J). \ea 
In the above we have used the compact notation:
$$ \ip{K ^A \{Z,Z^*y,v\} }{X K \{ W , x, u \} } := \sum_{j=1}^d \ip{K ^A \{Z,Z_j^*y,v\} }{X_j K \{ W , x, u \} }, $$ and we will continue to use this throughout.
Calculate:
\ba & &  \ip{K ^A \{Z , Z^*y , v \}}{X K\wxu}  = \ip{ K^A \{ Z, Z^* y , v \} }{K^A \{ W , W^* x , u \} } \nn \\
& & -\ip{ K^A \{ Z, Z^* y , v \} }{ \mbf{A} } \ipcm{ u}{A^\dag (W) ^* x} \nn \\
& = & \ipcn{y}{ZK^A (Z,W) [ v^* u ] W^* x} -\ipcn{y}{Z \mbf{A} (Z) v^*}\ipcm{ u^*}{ A^\dag (W) ^*x} \nn \\
& = & \ipcn{y}{\left( K^A (Z,W) [vu^*] -vu^* +A^\dag (Z) v u^* A^\dag (W) ^* \right) x} \nn \\
& & -\ipcn{y}{(A^\dag (Z) - A ^\dag (0 _n ) ) v}\ipcm{u}{A^\dag (W)^*x} \nn \\
& = & \ipcn{y}{\left( K^A (Z,W) [vu^*] -vu^* +A^\dag (Z) v u^* A^\dag (W) ^* \right) x} \nn \\
& & - \ipcn{y}{A^\dag (Z) vu^* A^\dag (W)^*x} + \ipcn{y}{A(0_n)vu^* A^\dag (W)^* x}  \nn \\
& = & \ipcn{y}{\left( K^A (Z,W) [vu^* ] - K^A (0 _n , W) [vu^*] \right)x } \in \L (\J). \nn \ea In the above, note that $A^\dag (0_n) = A (0_n) = A(0) I_n = A_\emptyset I_n$ where $0 \in \B ^d _1 = \B ^d$, since $A^\dag$ is a free function. This proves that $X$ is a Gleason solution. 
To see that $X$ is contractive, again calculate on kernel vectors:
\ba  \| X K^A \{Z , y , v \} \| ^2  & = &  \ip{K^A \{ Z , Z^* y , v \} }{ K^A \{ Z , Z^* y , v \} } \nn \\
& &- \ip{K^A \{ Z , Z^* y , v \} }{\mbf{A}}\ipcn{A^\dag (Z)v}{y} \nn \\ 
& & - \ipcn{y}{A^\dag (Z) v} \ip{\mbf{A}}{K^A \{Z, Z^* y , v \}} + \| \mbf{A} \| ^2 \left| \ipcn{A^\dag (Z) v}{y} \right| ^2 \nn \\
& \leq & \ipcn{y}{ZK^A (Z,Z)[vv^*] Z^* y} - \ipcn{y}{(A^\dag (Z) - A(0_n) ) v} \ipcn{A^\dag (Z) v}{y} \nn \\
& & - \ipcn{y}{A^\dag (Z) v} \ipcn{(A^\dag (Z) - A(0_n)) v}{y} \nn \\
& & + (I - A (0) ^* A (0)) \ipcn{y}{A(Z) vv^* A^\dag (Z) ^* y} \nn \\
& = & \ipcn{y}{K^A (Z,Z) [vv^*] y} - \ipcn{y}{(vv^* - A^\dag (Z) vv^* A^\dag (Z) ^*) y} \nn \\
 & & -2 \ipcn{y}{A^\dag (Z) vv^* A^\dag (Z) ^* y} \nn  + \ipcn{y}{A(0_n) vv^* A^\dag (Z) ^* y} + \ipcn{A(0_n ) vv^* A^\dag (Z) ^* y}{y}  \nn \\
& & + (I - A(0) ^* A(0)) \ipcn{y}{A^\dag (Z) vv^* A^\dag (Z) ^* y} \nn \\
& = & \ipcn{y}{K^A (Z,Z) [vv^*] y} - \ipcn{y}{vv^* y} + \ipcn{y}{A(0_n) vv^* A^\dag (Z) ^* y}  \nn \\
 & &  \quad  + \ipcn{A(0_n ) vv^* A^\dag (Z) ^* y}{y} - A(0) ^* A(0) \ipcn{y}{A^\dag (Z) vv^* A^\dag (Z)^* y}. \label{referback} \ea
Observe that equality holds in the above if $\mbf{A}$ is extremal. Compare this to:
\ba & & \ip{K^A \{Z , y , v \}}{(I - K_0 ^A (K_0 ^A )^* ) K^A \{Z , y , v \}} \nn \\
& & =  \ipcn{y}{K^A (Z,Z) [vv^*] y} - \| K^A \{Z , y , v \} (0) \| ^2 \nn \\
& = & \| K^A \{Z , y , v \} \| ^2 - \left| \ipcn{y}{K^A (0_n , Z) v} \right| ^2 \nn \\
& = & \| K^A \{Z , y , v \} \| ^2 - \left| \ipcn{y}{\left( I_n - A(0_n) A^\dag (Z) ^* \right) v} \right| ^2 \nn \\
& = & \| K^A \{Z , y , v \} \| ^2 - \ipcn{y}{vv^* y} +2 \re{ \ipcn{y}{v}\ipcn{y}{A(0) A^\dag (Z)^* *} } \nn \\
& & - A(0) ^* A(0) \ipcn{y}{A^\dag (Z)^* v} \ipcn{A^\dag (Z) ^* v}{y}, \nn \ea which is the same (up to elementary manipulations) as Equation (\ref{referback}) above. This proves that $X^*X \leq I - K_0 ^A (K_0 ^A )^*$ so that $X$ is a contractive Gleason solution (which will be extremal if $\mbf{A}$ is).  
\end{proof}

\begin{thm} \label{uniqueGSthm}
Suppose that $A \in \scr{R} _d (\H ) $. Then $\mbf{A} := (L^* \otimes I_\H) A$ and $X := (L^* \otimes I_\H) | _{\scr{H} ^R (A)}$ are the unique contractive Gleason solutions for $A$ and $\scr{H} ^R (A)$, respectively.
\end{thm}
The proof uses similar arguments to those of \cite[Section 4]{JM}.
\begin{lemma}{ (\cite[Proposition 6.2]{JMfree})} \label{ClarkGSform}
The contractive Gleason solution $\check{\mbf{A}}=(L^* \otimes I _\H ) A$ is given by the formula
\ba \check{\mbf{A}} & = & \mc{F} ^R \Pi _A  ^* [I \otimes ] _A (I -A(0)) \nn \\
& =& \mc{U} _A (V^A) ^* K_0 ^{H_A} (I-A(0)), \nn \ea where $\Pi _A = \pi _A (L)$ is the row isometry obtained from the S-GNS representation of the free Clark measure $\mu _A$.
\end{lemma}
Recall that $\mc{U} _A := M^R _{(I-A ^\dag (Z))}$ is the unitary multiplier of $\scr{H} ^R _+ (H_A)$ onto $\scr{H} ^R (A)$, and $V^A$ is the row isometry on $\scr{H} ^L _+ (H_A)$ defined in Subsection \ref{CTofSGNS}.
\begin{proof}{ (of Theorem \ref{uniqueGSthm})}
    Let $\mbf{A}$ be any contractive Gleason solution for $A$. Define a linear map $D^* : \scr{H} ^R _+ (H_A) \rightarrow \scr{H} ^R _+ (H_A) \otimes \C ^d$ by:
\be D^* K^{H_A} \zyv := K^{H_A} \{ Z , Z^* y , v \} + M^R _{(I -A^\dag (Z) ) ^{-1}} \mbf{A} (I - A(0) ) ^{-1} \ipcn{v}{y}. \label{rowext} \ee 
Recall $\mc{U} _A  ^* := M^R _{(I-A^\dag (Z)) ^{-1}}$, the unitary right multiplier of $\scr{H} ^R (A)$ onto $\scr{H} ^R _+ (H_A)$. By construction,
\ba & & D^* \left( K \zyv - K \{ 0 _n , y, v \} \right) = K^{H_A} \{ Z , Z^* y , v \}  \nn \\
& = & (V^A) ^* \left( K \zyv - K \{ 0 _n , y, v \} \right). \nn \ea

We claim that $D^*$ is a contraction:
\ba & & \| D^* K^{H_A} \zyv \| ^2 =  \| K^{H_A} \{ Z , Z^*y , v \} \| ^2  \nn \\
& & + \ip{ K^{H_A} \{Z , Z^* y , v \} }{\mc{U} _A ^* \mbf{A}} (I - A(0) ) ^{-1} \ipcn{v^*}{y} +c.c. \nn \\
& & + \ipcn{y}{v} (I -A(0) ^* ) ^{-1} \ip{\mbf{A}}{\mbf{A}} (I-A(0))^{-1} \ipcn{v}{y} \nn \\
& \leq & \| K^{H_A} \{ Z ,   Z^*y , v \} \| ^2 + \ip{ K^{H_A} \{ Z, Z^* y, (I-A^\dag (Z) ) ^{-1}v  \} }{\mbf{A}} (I-A(0))^{-1} \ipcn{v}{y} \nn \\
& &  +c.c. + K^{H_A} (0,0) \ipcn{y}{vv^*y},   \label{part1} \ea
where $c.c.$ denotes the complex conjugate of the previous term. 
The cross-term becomes:
\ba && \ip{K^{H_A} \{ Z, Z^* y,  (I-A^\dag (Z) ) ^{-1}v \} }{\mbf{A}} (I-A(0))^{-1} \ipcn{v}{y} \nn \\
& = & \ipcn{y}{(I-A(0_n) ) ^{-1} (A^\dag (Z) - A (0_n) ) (I -A^\dag (Z)) ^{-1} v}\ipcn{v}{y} \nn \\
& = & \frac{1}{2} \ipcn{y}{(H_A ^\dag (Z) -H_A ^\dag (0_n) ) (vv^*) y }. \nn \ea
It follows that Equation (\ref{part1} ) becomes:
\ba & & \| K^{H_A} \{ Z , Z^* y , v \} \| ^2  +\frac{1}{2} \ipcn{y}{\left( H_A ^\dag (Z) vv^* + vv^* H_A ^\dag  (Z) ^* \right) y} \nn \\
& =  & \| K^{H_A} \{ Z, y, v \} \| ^2. \nn \ea  
This proves that $D^*$ is a contractive extension of $(V^A) ^*$ so that $D$ is a row contractive extension of $V^A$  (by, for example, \cite[Lemma 2.3]{JM}). However $V^A$ is a row isometry and has no non-trivial extensions. Hence, $D = V^A$, and Equation (\ref{rowext}) and Lemma \ref{ClarkGSform} then imply that $\mbf{A} = (L^* \otimes I_\H) A =\check{\mbf{A}}$.
It follows also that $X$ is unique, by Proposition \ref{GSforHBandB}, so that $X = (L^* \otimes I_\H) | _{\scr{H} ^R (A)}$.
\end{proof}

If $B \in \scr{L} _d (\H)$, the formula is similar:
\be \mbf{B} = M^L _{(I-B(Z))} (V^B)^* K_0 ^{H_B} (I - B(0)). \label{GSforB} \ee

\section{Column Extreme}

Recall that $B \in \scr{L} _d (\H)$ or $A \in \scr{R} _d (\H)$ is said to be quasi-extreme if 
$F^2 _0 (\mu _B ) = F^2 (\mu _B )$, \emph{i.e.} if and only if 
$$ I \otimes \H \subseteq F^2 _0 (\mu _B ), $$ see \cite{JMfree} and \cite[Definition 3.19]{JM}. This concept of quasi-extreme was first introduced for contractive scalar multipliers of the Drury-Arveson space in \cite{Jury2014AC}, extended to operator-valued multipliers $b \in \scr{S} _d (\H)$ in \cite{JM}, and to the `rectangular setting' of arbitrary $b \in \scr{S} _d (\H, \J)$ in \cite{MM-dBR}. (Here $\scr{S} _d (\H, \J)$ denotes the Schur class of contractive operator-valued multipliers between vector-valued Drury-Arveson spaces.) The main result of \cite{JMqe} shows that a more descriptive name for this property could be \emph{column extreme} (CE), and we will use this new terminology for the remainder of this paper.

\begin{defn}
    A Schur class $B \in \scr{L} _d (\H, \J)$ is \emph{column extreme} (CE) if there is no non-zero 
$A \in \scr{L} _d (\H , \J)$ so that the column:
$$ \bpm B \\ A \epm \in \scr{L} _d (\H , \J \otimes \C ^2 ), $$ is also Schur class. Column extreme for the right Schur class is defined analogously.
\end{defn}
\begin{remark}
Observe that the definition of column extremity can be recast as follows: $B$ is column extreme if and only if the only multiplier $A$ satisfying the inequality
\begin{equation}
M^{L*}_AM^L_A\leq I-  M^{L*}_BM^L_B
\end{equation}
is $A=0$. The existence of such $A$ for given $B$ was considered by Popescu \cite{PopEntropy}, who showed that a nonzero $A$ exists if and only if $e(I-  M^{L*}_B M^L_B)>-\infty$, where $e(\cdot)$ is the so-called {\em entropy} of a multi-analytic Toepliz operator as defined in  \cite{PopEntropy}. However, it seems to be difficult to compute the entropy for arbitrary $B$ (or even to decide if it is finite or not). Regarding the equivalences in Theorem~\ref{CEthm} below, it is not hard to see from the definition of the entropy invariant, that $e(I-  M^{L*}_B M^L_B)=-\infty$ is equivalent to our condition (5), so that the equivalence of (1) and (5) is essentially contained in \cite[Corollary 1.2]{PopEntropy}. 
\end{remark}

In this general `rectangular' setting, it will often be convenient to consider the \emph{square completion} $[B]$, of $B$: The above column-extreme property is clearly invariant under conjugation by isometries; a given $B \in \scr{L} _d (\H, \J)$ is CE if and only if $B' = V B W^*$ is CE, where $W: \H \rightarrow \H '$, and $V:\J \rightarrow \J '$ are fixed onto isometries.
It follows that we can assume, without loss of generality, that $\H \subseteq \J$ or $\J \subsetneq \H$, and complete $B$ to a `square' $[B] \in \L _d (\J)$ or in $\L _d (\H)$, respectively by adding columns or rows of zeros:
$$ [B] := \left\{ \begin{array}{cc} \bbm B & 0 _{\J \ominus \H} \ebm &  \H \subsetneq \J \\
\bbm B \\ 0 _{\H \ominus \J} \ebm & \J \subsetneq \H. \end{array} \right. $$

\begin{remark}
Observe that if $\H \subsetneq \J$ then $\scr{H} ^L (B) = \scr{H} ^L ([B])$ so that the unique contractive Gleason solution for $B$ is given by $[\mbf{B} ] | _{\H}$, where $[\mbf{B} ]$ is the unique contractive Gleason solution for $[B]$. It is clear that $\mbf{B}$ is extremal if and only if $[\mbf{B}]$ is extremal in this case.

In the second case where $\J \subsetneq \H$ we have that 
$\scr{H} ^L ([B] ) = \scr{H} ^L (B) \bigoplus \left(F^2 _d \otimes (\H \ominus \J) \right)$, and the unique contractive Gleason solution for $[B]$ is given by:
$$ [\mbf{B} ] = (L^* \otimes I_\H) \bpm B \\ 0 _{\H \ominus \J }  \epm = \bpm \mbf{B} \\ \mbf{0} _{\H \ominus \J} \epm, $$ where $\mbf{0} _{\H \ominus \J } : \H \ominus \J \rightarrow \scr{H} ^L (B) \otimes \C ^d$ maps everything to the zero element. In this case it is clear that 
$$ \mbf{B} = \bpm I, & 0 \epm [\mbf{B} ], $$ and it follows as before that $\mbf{B}$ is extremal if and only if $[\mbf{B}]$ is extremal.
\end{remark}

\begin{thm} \label{CEthm}
Given $B \in \scr{L} _d (\H , \J)$, the following are equivalent:
\bn
    \item $B$ is column extreme.
    \item The unique contractive Gleason solution $\mbf{B} = R^* B : \H \rightarrow \scr{H} ^L (B) \otimes \C ^d$, for $B$ is extremal.
    \item The unique contractive Gleason solution $X = R^* | _{\scr{H} ^R (B)}$ for $\scr{H} ^R (B)$ is extremal, and $\H = \mr{supp} (B)$.
    \item $K_0 ^{H_{[B]}} (I - B(0))  \H \subseteq \ran{V^{[B]}}$.
    \item $B$ has the Szeg\"{o} extremal property: $I \otimes (I-B(0)) \H \subseteq F^2 _0 (\mu _{[B]} )$.
    \item There is no non-zero $\H$-valued constant function $H\equiv h \in \scr{H} ^L _+ (H_{[B]})$.
    \item There is no non-zero $h \in \H$ so that $Bh \in \scr{H} ^L (B)$. 
\en
If $B =[B]$ is square, then the above are equivalent to:
\bn
    \item[(8)] $\pi _{\mu _B } (L)$ (equivalently $V^B$) is a Cuntz row isometry.
\en
\end{thm}
In the above, recall that $F^2 _0 (\mu _B) = \bigvee _{\alpha \neq \emptyset} L^\alpha \otimes I_\H \subseteq F^2 (\mu _B)$, and the support of $B$ was defined in Equation (\ref{support}).

\begin{remark}
In the classical (single-variable, scalar-valued) setting, the equivalent statements in the above theorem recover several characterizations of extreme points of
the Schur class, $\scr{S}$, of contractive analytic functions on $\D$:
\begin{thm*}
Given $b \in \scr{S}$, the following are equivalent:
\bn
    \item[(0)] $b$ is an extreme point.
    \item $b$ is column extreme.
    \item $S^*b$ is extremal, \emph{i.e.} $\|S^* b \| _{\scr{H} (b)} = 1 - |b(0)| ^2$.
    \item $X := S^* | _{\scr{H} (b)}$ is extremal, \emph{i.e.}, $X^* X = I - k_0 ^b (k_0 ^b) ^*$.
    \item[(5)] $H^2 (\mu _b) = H^2 _0 (\mu _b)$. 
    \item[(7)] $b$ does not belong to $\scr{H} (b)$.
    \item[(8)] All Clark perturbations of $S^* | _{\scr{H} (b)}$ are unitary.
\en
\end{thm*}
Conditions $(1)$ and $(2)$ are characterizations of extreme points of $\scr{S}$ which follow from results of Sarason, \cite[Chapter III, Chapter IV]{Sarason-dB}. The spaces $F^2 (\mu _B)$, and $F^2 _0 (\mu _B)$ are multi-variable non-commutative analogues of the spaces $H^2 (\mu _b)$, and $H^2 _0 (\mu _b)$, the closure of the analytic polynomials and closed span of the non-constant analytic monomials (\emph{resp}.) in $L^2 (\mu _b)$, for $b \in \scr{S}$. In the classical setting the condition that $H^2 _0 (\mu _b) = H^2 (\mu _b )$ is equivalent to $b$ being an extreme point of the Schur class.  This follows from the Szeg\"{o}-Kolmogoroff-Kre\u{\i}n distance formula for the distance from the constant function $1$ to $H^2 _0 (\mu _b)$, the formula for the Radon-Nikodym derivative of $\mu _b$ with respect to normalized Lebesgue measure, and other classical facts, see \cite[Chapter 4, Chapter 9]{Hoff}, and the discussion in \cite[Section 1]{JMfree}. Item $(7)$ is again a result of Sarason \cite[Chapter IV]{Sarason-dB}, and the final item is equivalent to the well-known fact that $b$ is extreme if and only if all of the Clark perturbations of the restricted backward shift $S^* | _{\scr{H} (b)}$ are unitary, see \emph{e.g.} \cite{Martin-dB}. Corollary \ref{extremecor} will prove that any column-extreme $B \in \scr{L} _d$ or $\scr{R} _d$ is necessarily an extreme point, whether the converse holds is an open problem.
\end{remark}

We will need the following free or NC analogue of a result from vector-valued RKHS theory, \cite[Theorem 10.17]{Paulsen-rkhs} (see the proof of \cite[Proposition 5.1]{JMqe}):
\begin{lemma} \label{NCRKHSfact}
    Let $\H _{nc} (K)$ be a vector-valued NC-RKHS on a NC set $\Om$. A (vector-valued) free NC function $f$ on $\Om$ belongs to
$\H _{nc} (K)$ if and only if $\la ^2 K(Z,W) - f(Z) (\cdot ) f(W) ^* \geq 0$ is a (operator-valued) CPNC Kernel on $\Om$ for some $\la ^2 >0$. The norm
of $f$ is the infimum of all such $\la$.
\end{lemma}
\begin{lemma} \label{confunLem}
    If $B \in \scr{L } _d (\H, \J)$ and $C := [B] \in \scr{L } _d (\K)$, then $Bh \in \scr{H} ^L (B)$ for $h \in \H$ if and only if $F \equiv (I - C(0) ^* ) ^{-1} h \in \H$ is a constant function in $\scr{H} ^L _+ (H_C)$. 
\end{lemma}
\begin{proof}
    If $Bh \in \scr{H} ^L (B)$ set $f:= (I - C(0) ^* ) ^{-1} h \in \H$. Then,
\ba (K_0 ^B f ) (Z) & = &I_n \otimes f - B(Z) (I_n \otimes B(0) ^* (I - C(0) ^* ) ^{-1} h) \nn \\
& = & I_n \otimes f - B(Z) C(0) ^* (I - C(0) ^* ) ^{-1} h \nn \\
& = & I_n \otimes f - B(Z)(I -C(0) ^* ) ^{-1} h + B(Z)h \nn \\
& = & (I_n - B(Z) )f + B(Z) h. \nn \ea
Since $B(Z) h = B(Z) (I_n \otimes h) \in \scr{H} ^L (B)$, we conclude that $F \equiv f \in \H$ belongs to $\scr{H} ^L _+ (H_C)$ (it is the image of $(I - B) f$ under the canonical unitary multiplier). 

This argument is reversible: if $H \in \scr{H} ^L _+ (H_C)$ is such that $H\equiv h \in \H$, \emph{i.e.} 
$H(Z) = I_n \otimes h$, then $M^L _{(I-B(Z))} H = F \in \scr{H} ^L (C)$, and also
$$ K_0 ^B (Z) h = I_n \otimes h - B(Z) (I_n \otimes B(0) ^* h), $$ so that 
$$ (K_0 ^B - F) (Z) = B(Z) (I -B(0) ^* ), $$ and we conclude that $B (I-C(0) ^* ) h =B(I-B(0) ^* ) h\in \scr{H} ^L (C)$. If $\J \supseteq \H$ so that $[B] \in \scr{L} _d (\J )$, then $\scr{H} ^L (C) = \scr{H} ^L (B)$. Otherwise if $\H \supsetneq \J$ then $\scr{H} ^L (C) = \scr{H} ^L (B) \oplus F^2 _d \otimes (\H \ominus \J)$. Since $C(Z) h = B(Z) h$ for any $h \in \H$, and $B(Z) \in \C ^{n\times n} \otimes \L (\H , \J)$, Lemma \ref{NCRKHSfact} implies that there is a $ \la ^2 > 0$ so that
\ba  & & \bpm B(Z) g (B(W) g ) ^* & 0 \\ 0 & 0 _{\H \ominus \J} \epm = C(Z)g (C(W)g)^* \nn \\
& \leq & \la ^2 K^C (Z,W)  = \la ^2 \bpm K^B (Z,W) & 0 \\ 0 & K(Z,W) \otimes (\H \ominus \J ) \epm, \nn \ea where $g:=(I-B(0) ^* ) h = (I-C(0)^*) h \in \H$. Comparing top left entries, Lemma \ref{NCRKHSfact} again implies that $B(Z) (I-B(0)^*)h =Bg  \in \scr{H} ^L (B)$.
\end{proof} 
The proof of equivalence of the first two items is the most involved, so we will first establish the equivalence of the remaining items. 
\begin{proof}{ (of equivalence of items $(2)-(10)$ in Theorem \ref{CEthm})}
To simplify notation, we will write $C =[B] \in \scr{L} _d (\K)$ (with $\K = \H$ or $\J$).
\paragraph{$(2) \Leftrightarrow (3)$} This was proven as part of Proposition \ref{GSforHBandB} which relates Gleason solutions for $\scr{H} ^L (B)$ and Gleason solutions for $B$. Also note that if $\mr{supp}(B) \subsetneq \H$, then $B$ will have a matrix representation of the form:
$$ \bpm B ' , & 0 \epm, $$ so that for any Schur $A ' \in \scr{L} _d $, 
$$ \bpm B' & 0 \\ 0 & A ' \epm =: \bpm B \\ A \epm \in \scr{L} _d (\H \oplus \C, \J \oplus \C ), $$ is Schur and $B$ is not CE in this case.

\paragraph{$(2) \Leftrightarrow (4)$} There are two cases to consider. If $C=[B] \in \scr{L} _d (\H)$ then $\scr{H} ^L (B) \oplus F^2 _d \otimes (\H \ominus \J) = \scr{H} ^L (C)$ and $\mbf{B}$ is extremal if and only if $\mbf{C}$ is. If $\mbf{C}$ is extremal then
\ba I - C(0) ^* C(0) & = & \mbf{C} ^* \mbf{C}  \nn \\
& = & (I - C(0) ^*) (K_0 ^{H_C}) ^* V^C  (V^C) ^* K_0 ^{H_C} (I -C(0)) \nn \\ 
& \leq & (I - C(0) ^*) K^{H_C} (0 , 0 ) (I - C(0) ) \nn \\
& = & I - C(0) ^* C(0), \nn \ea and this happens if and only if $V^C$ is a co-isometry. This establishes the equivalence in this case. Alternatively, if $\H \subseteq \J$ then $\mbf{B} = \mbf{C} | _\H$ so that $\mbf{B}$ will be extremal if and only if 
\ba I - B(0) ^* B(0) & = &  \mbf{B} ^* \mbf{B} \nn \\
& = & P_\H \mbf{C} ^* \mbf{C} P_\H \nn \\
& = & P_\H (I - C(0) ^*) (K_0 ^{H_C}) ^* V^C  (V^C) ^* K_0 ^{H_C} (I -C(0)) P_\H \nn \\
& \leq & P_\H (I - C(0) ^* C(0) ) P_\H = I - B(0) ^* B(0). \nn \ea  The above holds if and only if $K_0 ^{H_C} (I - C(0)) \H = K_0 ^{H_C} (I -B(0)) \H \subseteq \ran{V^C}$, and this proves the equivalence in the second case.

\paragraph{$(4) \Leftrightarrow (5)$} This follows immediately from the fact that the Cauchy transform $\mc{C} ^L : F^2 (\mu _C) \rightarrow \scr{H} ^L _+ (H_C)$ is an onto isometry which intertwines $\Pi ^C = \pi _C (L)$ and $V^C$, which takes $F^2 _0 (\mu _C)$ onto $\ran{V^C}$, and which maps $I \otimes g \in F^2 (\mu _C)$ to $K_0 ^{H_C} g$, see Proposition \ref{CTthm}.

\paragraph{$(4) \Rightarrow (6)$} Assume that $(6)$ does not hold so that there is a constant $\H-$valued function $H \equiv h \in \H$ in $\scr{H} ^L _+ (H_C)$. Set $f := (I - B(0) ^* ) h \in \H$ so that $h = (I - C(0) ^* ) ^{-1} f$. If $(4)$ holds then $K_0 ^{H_C} (I - C(0)) f \in \ran{V^C} \perp H$ so that 
$$ 0 = \ip{K_0 ^{H_C} (I - B(0) ) f }{H} = \ip{(I - C(0)) f}{(I-C(0) ^* ) ^{-1} f} = \| f \| ^2. $$ We conclude that $f =0$ so that $h =0$ (since $B(0) ^*$ is a pure contraction). This shows that $(4)$ cannot hold.

\paragraph{$(6) \Leftrightarrow (7)$} This equivalence is an immediate consequence of Lemma \ref{confunLem}.

\paragraph{$(7) \Rightarrow (4)$} Our proof will be a bit circuitous: First consider the following condition $(4)'$: If $F \in \scr{H} _+ ^L (H_C)$ is constant valued, $F \equiv f \in \K$, then 
$$ P_\H (I - B(0) ^* ) f =0. $$ We claim that $(4)' \Rightarrow (4)$. Condition $(4)'$ implies that $$ P_\H (I- B(0) ^*) (K_0 ^{H_C})^* (I - V V ^*) =0, $$ and taking the adjoint of this expression gives:
$$ (I - VV^* ) K_0 ^{H_C} (I - B(0) ) P_\H =0, $$ which is condition $(4)$. It remains to show that $(7) \Rightarrow (4)'$, and this will be accomplished by demonstrating the contrapositive. If $(4)'$ does not hold then there is a $f \in \K$ so that $F (Z) = I_n \otimes f$ belongs to $\scr{H} _+ ^L (H_C)$, and $P_\H (I - B(0) ^* ) f \neq 0$. By Lemma \ref{confunLem}, $(I - C(0) ^* ) f = g$ is such that $Cg \in \scr{H} ^L (C)$.  There are two cases: If $\H \subseteq \J$ then $\scr{H} ^L (C) = \scr{H} ^L (B)$. In this case $C g = B P_\H g \in \scr{H} ^L (B)$, and by assumption, $P_\H g = P_\H (I-B(0) ^* ) f \neq 0$. We conclude that $(7)$ does not hold in this case. In the second case $\J \subseteq \H$, and $g = (I - B(0)^*) f \in \H$ is such that $Cg = B g \in \scr{H} ^L (C) = \scr{H} ^L (B) \oplus F^2 _d \otimes (\H \ominus \J)$. Since $B(Z) \in \C ^{n\times n} \otimes \L (\H, \J)$, one can apply Lemma \ref{NCRKHSfact} (as in the proof of Lemma \ref{confunLem}) to show that $Bh \in \scr{H} ^L (B)$, and again $(7)$ does not hold. Hence $(7) \Rightarrow (4)' \Rightarrow (4)$.

Assuming now that $B =[B] = C \in \scr{L} _d (\H)$ is square, item $(6)$ is equivalent to the statement that $V^C$ is a Cuntz (onto) row isometry, and since $V^C$ is unitarily equivalent to $\Pi ^C = \pi _C (L)$ via Cauchy transform, it follows that $(6) \Leftrightarrow (8)$. 
\end{proof}

The proof of $(1) \Leftrightarrow (2)$ is the free and operator-valued extension of the main result of \cite{JMqe}, and the argument is formally analogous. 

\begin{proof}{ ($(1) \Rightarrow (2)$ of Theorem \ref{CEthm})}
Suppose $B \in \scr{L} _d (\H , \J)$ is not column-extreme so that there is a non-zero $A \in \scr{L} _d (\H , \J)$ so that the two-component column $C:= \bpm B \\ A \epm$ is Schur. Without loss of generality, we can assume that $ 0 \neq A _\emptyset = A(0) \in \L (\H, \J)$. The argument is as in \cite[Lemma 5.2]{JMqe}: If $A_\emptyset =0$ choose $\alpha \in \F ^d$ of minimal length so that $A_\alpha \neq 0$, and set: $\wt{A} := (L^\alpha ) ^* \otimes I_\J A$. Then 
$$ \bpm B \\ \wt{A} \epm = \bpm I & 0 \\ 0 & (L _\alpha ) ^* \otimes I_\J \epm \bpm B \\ A \epm, $$ is also Schur and satisfies $\wt{A} _\emptyset = A _\alpha \neq 0$.
The unique contractive Gleason solution, $\mbf{C}$ for $C$ is:
$$ \mbf{C} = \bpm R^* \otimes I _\J , & R^* \otimes I_\J \epm \bpm  B \\ A \epm = \bpm \mbf{B} \\ \mbf{A} \epm, $$ where $\mbf{B}, \mbf{A}$ are the unique contractive Gleason solutions for $B,A$. Observe that 
$$ \mbf{C} ^* \mbf{C} \leq I - C(0) ^* C(0) = I - B(0) ^* B(0) - A(0) ^* A(0) < I - B(0) ^* B(0), $$ since we can assume $A(0) \neq 0$.

Now we apply the argument of \cite[Proposition 5.1]{JMqe}: By Lemma \ref{NCRKHSfact}, since each $\mbf{C} _j h\in \scr{H} ^L (C)$ for any $h \in \H$, where $1 \leq j \leq d$, there is a $t_j >0$ so that as CPNC kernels, 
\ba  (C_j (Z) h) (C_j(W)h) ^*  &= & \bpm (B_j (Z)h) (B_j (W)h)^* & * \\ * & * \epm \nn \\
& \leq &  t_j ^2 K^C (Z,W) = t_j ^2 \bpm K^B (Z,W) & * \\ * & * \epm, \nn \ea and one can take $t_j := \| C_j h \|$.
It follows that for any $h \in \H$, 
\ba \| \mbf{B} h \| ^2 & = &\ip{h}{\mbf{B} ^* \mbf{B} h} \nn \\
& \leq & \sum _{j=1} ^d t_j ^2 \nn \\
& = & \ip{h}{\mbf{C} ^* \mbf{C} h} \nn \\
& \leq & \ip{h}{(I-C(0) ^* C(0)) h } \nn \\
& = & \ip{h}{(I - B(0) ^* B(0) - A(0) ^* A(0) ) h}. \nn \ea This proves that 
$$ \mbf{B} ^* \mbf{B} < I - B(0) ^* B(0), $$ so that $\mbf{B}$ is not extremal. 
\end{proof}

The proof of $(2) \Rightarrow (1)$ will employ the colligation and transfer-function theory of \cite{Ball-Fock}. We briefly recall the pertinent facts: A \emph{colligation} is any contractive linear map: 
$$ U := \bpm A_1 & B_1 \\ \vdots & \vdots \\ A_d & B_d \\ C & D \epm =: \bpm A & B \\ C & D \epm : \bpm \K \\ \H \epm \rightarrow \bpm \K \otimes \C ^d \\ \J \epm. $$ 
The \emph{transfer-function} of the contractive colligation $U$ is the function $B_U$ defined on the free unit ball by:
$$ B _U (Z) := D + C (I -ZA) ^{-1} ZB \in \C ^{n\times n} \otimes \L (\H , \J); \quad \quad Z \in \B ^d _n, $$ where $ZA := Z_1 A_1 + ... Z_d A_d.$ The theory of \cite{Ball-Fock} shows that a free function $B$ on $\B ^d _\N$ belongs to the left free Schur class if and only if $B=B_U$ is the transfer function of some contractive colligation $U$ (see \cite[Theorem 3.1]{Ball-Fock}). Moreover, any $B \in \scr{L} _d (\H, \J)$ always has the (left) \emph{canonical deBranges-Rovnyak colligation}
$$ U_{dBR} := \bpm A_{dBR} & B_{dBR} \\ C_{dBR} & D _{dBR} \epm $$ constructed by choosing $\K := \scr{H} ^L (B)$ and 
$$ A_{dBR} := R^* | _{\scr{H} ^L (B)}, \quad B_{dBR}:= R^* B, \quad C_{dBR}:= (K_0 ^B ) ^*, \ \mbox{and} \ D_{dBR}:= B(0), $$ (so that $B$ is recovered as the transfer function of this colligation) see \cite[Theorem 4.3]{Ball-Fock}.
Similarly, if $A \in \scr{R} _d (\H, \J)$ then $A(R) = M^R _{A^\dag (Z)}$ is such that the free holomorphic function $A ^\dag (Z)$ can be recovered as the transfer function of the (right) canonical deBranges-Rovnyak colligation given by choosing $\K := \scr{H} ^R (A)$, and 
$$ A_{dBR} := L^* | _{\scr{H} ^R (A)}, \quad B_{dBR} := L^* A, \quad C_{dBR} := (K_0 ^A ) ^*, \ \mbox{and} \ D_{dBR} := A(0). $$ 

\begin{proof}{ (of $(2) \Rightarrow (1)$)}
We give the proof for right free multipliers $A \in \scr{R} _d (\H , \J)$. Assuming $\mbf{A}$ is not extremal, we choose $ 0 \leq a_\emptyset \in \L(\H)$ satisfying:
$$ a_\emptyset ^2 = I - A(0) ^* A(0) - \mbf{A} ^* \mbf{A}. $$
As in the proof of Proposition \ref{GSforHBandB} one can calculate that,
\ba & & \ip{X^*X K^A \{Z,y,v \}}{K^A \wxu  }    \\
& = & \ipcn{y}{K^A (Z,W)[vu^*]  x} - \ipcn{y}{Z \mbf{A} (Z) v}\ipcm{u}{A^\dag(W) ^*x} \nn \\
& & -\ipcn{A^\dag(Z) ^* y}{v}\ipcm{W\mbf{A} (W) u}{x} + \ipcn{y}{A^\dag (Z) \left( vu^* \otimes \mbf{A} ^* \mbf{A} \right) A^\dag(W) ^* x}. \nn \ea 
In the above $K^A \zyv$ is a bounded linear map from $\J$ into $\scr{H} ^R (A)$, so that the above inner product is $\L (\J )-$valued (we have omitted vectors to simplify the notation). Applying the definition of $a_\emptyset ^2$, the above becomes:
\ba 
&= & \ip{K^A \zyv}{K^A \wxu} - \ipcn{y}{(vu^* - A^\dag(Z) v u^* A^\dag (W) ^*) x} \nn \\
& & - \ipcn{y}{(A^\dag(Z) - A^\dag (0_n)) vu^* A^\dag (W)^* x} - \ipcn{y}{A^\dag(Z) vu^* (A^\dag(W) ^* - A^\dag (0_m) ^* ) x} \nn \\
& & + \ipcn{y}{A^\dag (Z) \left( vu^* \otimes (I - A(0) ^* A(0) - a_\emptyset ^2 ) \right) A^\dag (W) ^* x } \nn \\
& = & \ip{K^A \zyv}{K^A \wxu} - \ipcn{y}{A^\dag (Z) (vu^* \otimes a_\emptyset ^2) A^\dag (W) ^* x} \nn \\
&  & - \ipcn{y}{vu^* x} + \ipcn{y}{A(0_n) vu^*A^\dag (W) ^* x}  \label{line1} \\
& & +\ipcn{y}{A^\dag (Z) vu^* A(0_m) ^* x} - \ipcn{y}{A^\dag (Z) (vu^* \otimes A(0) ^* A(0) ) A^\dag (W) ^* x}. \label{line2} \ea
On the other hand, one can calculate that (up to a change of sign) line (\ref{line1}) $+$ line (\ref{line2}) in the above are equal to:
$$ \ip{K_0 ^A (K_0 ^A) ^* K^A \zyv}{K^A \wxu}, $$ and it follows that 
\ba & & \ip{X^*X K^A \zyv}{K^A \wxu} =  \ip{(I-K_0 ^A (K_0 ^A) ^*) K^A \zyv}{K^A \wxu} \label{magic} \\
& & - \ipcn{y}{A^\dag (Z) (vu^* \otimes a_\emptyset ^2) A^\dag (W) ^* x}. \nn  \ea
If we define the $\L (\J) -$valued CPNC kernel:
$$ G^A (Z,W) [P] := A^\dag (Z) [\cdot] \otimes a_\emptyset ^2 A^\dag (W) ^* , $$ then Equation (\ref{magic}) implies that $G^A \leq K^A$ as CPNC kernels so that, by Lemma \ref{NCRKHSfact}, 
$$ A a _\emptyset  : \H \rightarrow \scr{H} ^R (A), $$ (where $(Aa_\emptyset) (Z) = A^\dag (Z) a_\emptyset$). Moreover, Equation (\ref{magic}) further implies that 
\be I -X ^* X = K_0 ^A (K_0 ^A ) ^* + A a _\emptyset  (A a_\emptyset  ) ^*. \label{exactGS} \ee This is the appropriate analogue of the formula from \cite[Proposition 3.2]{JMqe}. To complete the proof, we apply the transfer function theory of \cite{Ball-Fock}. We define:
$$ U := \bpm X & \mbf{A} \\ (K_0 ^A) ^* & A(0) \\ -(A a_\emptyset  ) ^* & a _\emptyset \epm : \bpm \scr{H} ^R (A) \\ \H \epm \rightarrow \bpm \scr{H} ^R (A) \otimes \C ^d \\ \J \otimes \C ^2 \epm. $$
The top $2\times 2$ block of $U$ is the canonical deBranges-Rovnyak colligation with transfer function equal to $A$. It follows that if the above $U$ is contractive, then its transfer function will have the form: $$ \bpm A \\ a \epm \in \scr{R} _d (\H , \J \otimes \C ^2 ), $$ for some non-zero $a \in \scr{R} _d (\H , \J)$. We will prove that, in fact, $U$ is an isometry:
\be U^* U = \bpm X^* X + K_0 ^A (K_0 ^A ) ^* + A a_\emptyset  (A a_\emptyset  ) ^* & X^* \mbf{A} + K_0 ^A A(0) - A a _\emptyset ^2 \\ 
* & \mbf{A} ^* \mbf{A} + A(0) ^* A(0) + a_\emptyset ^2 \epm. \ee 
By previous formulas the diagonal entries are equal to $I _{\scr{H} ^R (A)}$ and $I _\H$, respectively, and it remains to show that the top right (and hence also the bottom left) component vanishes. 
This can be verified as follows:
\ba & & \ip{K^A \zyv f }{X^* \mbf{A} h}  =  \ip{K^A \{Z , Z^* y , v \} f - \mbf{A} \left( v \otimes I_\H , A^\dag (Z) ^* y \otimes f \right) }{\mbf{A} h} \nn \\
& = & \left( y \otimes f , (A^\dag (Z) - A(0 _n ) ) v\otimes h \right) - \left( (v \otimes I_\H , A^\dag(Z) ^* y \otimes f ), (I - A(0) ^* A(0) - a_\emptyset ^2 ) h \right) \nn \\
& = & \left( y \otimes f , \left( A^\dag (Z) - A(0_n) - A^\dag (Z) + A^\dag (Z)  \otimes (A(0) ^* A(0) + a_\emptyset ^2)   \right) v\otimes h \right). \nn \ea 
On the other hand, 
\ba & & \ip{K^A \zyv f }{K_0 ^A A(0) h - A a_\emptyset ^2 h } \\
& =& \left( y \otimes f , v \otimes A(0) h \right) - \left( y \otimes f , A^\dag (Z) v \otimes A(0) ^* A(0) h \right) - \left( y \otimes f , A^\dag (Z) v \otimes a_\emptyset ^2 h \right). \nn \ea 
Adding these expressions together gives $0$, which proves the off-diagonal component vanishes.
\end{proof}

\begin{cor} \label{extremecor}
    If $B \in \scr{L} _d (\H , \J)$ is column extreme, then it is an extreme point.
\end{cor}

\begin{proof}
This is the same contrapositive proof as in \cite[Corollary 1.2]{JMqe}: If $B$ is not extreme then there is a non-zero $A \in \scr{L} _d (\H, \J)$ so that both $B \pm A$ are Schur class, which implies:
$$ (M^L _{B-A} ) ^* M_{B-A} ^L  \leq I, \quad \mbox{and} \quad (M^L _{B+A} ) ^* M_{B+A} ^L  \leq I. $$
Averaging these inequalities gives:
$$ (M^L _A ) ^* M^L _A + (M_B ^L ) ^* M_B ^L \leq I, \ \mbox{\emph{i.e.},} \quad \bpm B \\ A \epm \in \scr{L} _d (\H , \J \otimes \C ^2), $$ so that $B$ is not column-extreme.
\end{proof}

The next two corollaries were established in the proof of $(2) \Rightarrow (1)$ of Theorem \ref{CEthm} above:
\begin{cor} \label{DomAhat}
Given $A \in \scr{R} _d (\H , \J)$, define $a_\emptyset \in \L (\H) ^+$ by the formula:
$$ a _\emptyset ^2 := I - A(0) ^* A(0) - \mbf{A} ^* \mbf{A} \geq 0. $$ Then for any $h \in \H$,
$A a_\emptyset h \in \scr{H} ^R (A)$, so that $A a _\emptyset : \H \rightarrow \scr{H} ^R (A)$. 
\end{cor}

Given $A \in \scr{R} _d (\H, \J)$, define $\mc{D} _A \subseteq \H$ as the linear space of all $h \in \H$ such that $Ah \in \scr{H} ^R (A)$, and let $\hat{A} : \mc{D} _A \rightarrow \scr{H} ^R (A)$ be the linear transformation $\hat{A} h := A h \in \scr{H} ^R (A)$. Here, we write $A = A(R)$ so that $(Ah) (Z) = A^\dag (Z) h$. The previous Corollary \ref{DomAhat} shows that $\ran{a_\emptyset} \subseteq \dom{\hat{A}}$.
\begin{cor}
Given $A \in \scr{R} _d (\H, \J)$, we have the identity:
$$ I -X^* X = K_0 ^A (K_0 ^A) ^* + (\hat{A} a_\emptyset )(\hat{A} a_\emptyset ) ^*. $$ 
\end{cor}
\begin{claim}
    The linear transformation $\hat{A}$ is closed.
\end{claim}

\begin{proof}
    Suppose $(h_n) \subset \mc{D} _A$, $h_n \rightarrow h \in \H$, and $Ah_n \rightarrow F \in \scr{H} ^R (A)$. It is easy to see that for any $Z \in \B ^d _n$, 
$$ F(Z) =  \lim _{n} A^\dag (Z) h_n = A^\dag (Z) h, $$ so that $F=Ah$, proving that $h \in \mc{D} _A$ and $\hat{A} h = F$. 
\end{proof}

Since $\hat{A}$ is a closed linear transformation, it follows by general facts that $\hat{A} ^* \hat{A}$ is densely-defined in the Hilbert space $\ov{\mc{D} _A}$ and positive semi-definite on a domain $\dom{\hat{A} ^* \hat{A}} \subseteq \mc{D} _A$ (which is a dense in $\mc{D} _A$).

\begin{prop} \label{domA}
    Given any $h \in \dom{\hat{A} ^* \hat{A}}$, 
    $$ a_\emptyset ^2 (I + \hat{A} ^* \hat{A}) h = h, $$ and $\dom{\hat{A}} = \mc{D} _A = \ran{a_\emptyset}$. Viewing $\hat{A}$ as a closed linear transformation from $\mc{D} _A \subseteq \ov{\mc{D}_A} \rightarrow \H$,
$$ a_\emptyset ^2 = (I _{\ov{\mc{D} _A}} + \hat{A} ^* \hat{A} ) ^{-1}. $$ 
\end{prop}
This is a free and operator-valued analogue of \cite[Lemma 3.3]{JMqe}.
\begin{proof}
    Given $h \in \dom{\hat{A} ^* \hat{A}} \subseteq \mc{D} _A$, we calculate $(I-X^* X) Ah \in \scr{H} ^R (A)$ in two different ways:
First, since $Ah \in \scr{H} ^R (A)$, $XAh =\mbf{A} h$, and 
\ba & & \ip{K^A \zyv g }{(I-X^*X)Ah}  = \left( y \otimes g , A^\dag (Z) v \otimes h \right) \nn \\
& & - \ip{K^A \{ Z , Z^* y, v \} g }{\mbf{A} h} + \ip{\mbf{A} \left( v \otimes I_\H , A^\dag (Z) ^* y \otimes g \right)}{\mbf{A} h} \nn \\
& = & \left( y \otimes g , A(0 _n ) v \otimes h \right) + \left( y \otimes g , A^\dag (Z) v\otimes \mbf{A} ^* \mbf{A} h \right). \nn \ea 
Equivalently, for any $Z \in \B ^d _n$,
$$ \left( (I-X^*X) A h \right) (Z) = A(0) (I_n \otimes h) + A^\dag (Z) (I_n \otimes \mbf{A} ^* \mbf{A} h ). $$ 
Secondly, apply the identity (\ref{exactGS}) to obtain:
\ba & & \left( K_0 ^A (K_0 ^A ) ^* Ah + A a_\emptyset (Aa_\emptyset ) ^* Ah \right) (Z) \nn \\
& = & K^A (Z , 0) A(0) h + A^\dag (Z) (I_n \otimes a_\emptyset ^2 \hat{A} ^* \hat{A} h) \nn \\
& = & A(0_n) (I \otimes h)  A^\dag (Z) \left( I_n \otimes a_\emptyset ^2 \hat{A} ^* \hat{A} h  - I_n \otimes A(0) ^* A(0) h \right). \nn \ea 
Equating these two expressions yields:
$$ A^\dag (Z) \left( I_n \otimes a_\emptyset ^2 \hat{A} ^* \hat{A} h - I_n \otimes A(0) ^* A(0) h \right) = A^\dag (Z) (I_n \otimes \mbf{A} ^* \mbf{A} h). $$ 
Using that $\mbf{A} ^* \mbf{A} = I - A(0) ^* A(0) - a_\emptyset ^2$ yields:
$$ a_\emptyset ^2 \hat{A} ^* \hat{A} h - A(0) ^* A(0) h = (I - A(0) ^* A(0) - a_\emptyset ^2)h, $$ and solving for $h$ gives:
$$ a_\emptyset ^2 (I + \hat{A} ^* \hat{A} ) h = h. $$ On the other hand, Corollary \ref{DomAhat} shows that $\ran{a_\emptyset} \subseteq \dom{\hat{A}} = \mc{D} _A$.  For any $g \in \dom{\hat{A}^* \hat{A}}$ and any $h \in \H$,
\ba \ip{g}{h} &=& \ip{ (I + \hat{A}^* \hat{A}) g}{a_\emptyset ^2 h} \nn \\
& = & \ip{g}{a_\emptyset ^2 h} + \ip{\hat{A} ^* \hat{A} g}{a_\emptyset ^2 h} \nn \\
& = & \ip{g}{a_\emptyset ^2 h} + \ip{\hat{A} g}{\hat{A} a_\emptyset h}. \nn \ea 
This can be re-arranged as:
$$ \ip{\hat{A}g}{\hat{A} a_\emptyset ^2 h} = \ip{g}{(I-a_\emptyset ^2 ) h}, $$ for any
$g \in \dom{\hat{A} ^* \hat{A}}$ and $h \in \H$.
Since $\hat{A}$ is a closed linear transformation, $\dom{\hat{A}^* \hat{A}}$ is a core for $\hat{A}$, and the above then implies that $\hat{A} a_\emptyset ^2 h \in \dom{\hat{A} ^*}$, and that 
$$ \hat{A}^* \hat{A} a_\emptyset ^2 h = (I - a_\emptyset ^2 )  h, $$ which is equivalent to
$$ (I +\hat{A} ^* \hat{A} ) a_\emptyset ^2 h = h, $$ and we conclude that $a_\emptyset ^2 = (I + \hat{A} ^* \hat{A} ) ^{-1}$.  This proves that $\ran{a_\emptyset ^2} = \dom{\hat{A}^* \hat{A}}$, and by the polar decomposition for closed operators, $\ran{a_\emptyset} = \dom{\sqrt{I +\hat{A}^* \hat{A}}} = \dom{\sqrt{\hat{A}^* \hat{A}}} = \dom{\hat{A}}$.
\end{proof}
\begin{cor}\label{Linv}
    Given $A \in \scr{R} _d (\H , \J)$, and $F \in \scr{H} ^R (A)$, we have that $L_j F \in \scr{H} ^R (A)$  if and only if $\mbf{A} _j  ^* F \in \dom{\hat{A}}$.
\end{cor}
\begin{proof}
The proof is formally identical to that of the commutative analogue of this result in \cite[Corollary 4.5]{Jury2014AC}. Consider, for $g \in \J$,
\ba \ip{K^A \zyv g}{ X_j ^* F } & = & \ip{K^A \{Z , Z_j ^* y , v \} g}{F} - \ip{\mbf{A} _j \left( v\otimes I_\H , A^\dag (Z) ^* y \otimes g\right) }{F} \nn \\
& =& \left( y \otimes g , Z_j F(Z) v \right) - \left( y \otimes g , A^\dag (Z) v\otimes \mbf{A} _j ^* F \right). \nn \ea 
Now $\mbf{A}_j ^* F \in \H$, so by the previous theorem, if $\mbf{A} _j ^* F \in \ran{a_\emptyset}$, then
$A \mbf{A}_j ^* F \in \scr{H} ^R (A)$ so that also $Z_j F(Z) \in \scr{H} ^R (A)$. Conversely if $Z_j F(Z) \in \scr{H} ^R (A)$, then the above formula shows that $A \mbf{A} _j ^* F \in \scr{H} ^R (A)$. 
\end{proof}

\begin{cor}
Given $A \in \scr{R} _d (\H, \J)$, if $Ah \in \scr{H} ^R (A)$ for every $h \in \H$, then $\scr{H} ^R (A)$ is $L$-invariant.
\end{cor}

\begin{cor} \label{equivLinv}
    Given any $A \in \scr{R} _d (\H, \J)$, we have that $\scr{H} ^R (A)$ is $L-$invariant if and only if 
$\ran{\mbf{A} ^*} \subseteq \ran{a_\emptyset} = \dom{\hat{A}}$. This happens if and only if $\hat{A}$ is densely-defined and there is a $0<r<1$ such that 
$$ r (I - A(0) ^* A(0) ) \leq (I + \hat{A} ^* \hat{A} ) ^{-1}, $$ or, equivalently, there is a $0<\rho <1$ so that 
$$ \mbf{A} ^* \mbf{A} \leq \rho (I - A(0) ^* A(0) ). $$  
\end{cor}

\begin{proof}
  By Corollary \ref{Linv}, $\scr{H} ^R (A)$ is $L-$invariant if and only if $\ran{\mbf{A} ^*} \subseteq \ran{a_\emptyset}$. By the Douglas Factorization Lemma, this happens if and only if there is a $\la ^2 >0$ so that 
$$ \mbf{A} ^* \mbf{A} ^* \leq \la ^2 a_\emptyset ^2 = \la ^2 (I - \mbf{A} ^* \mbf{A} - A(0) ^* A(0) ). $$
Re-arranging gives
$$\mbf{A} ^* \mbf{A} \leq \frac{\la ^2}{1+\la ^2} (I - A(0) ^* A(0) ). $$ 
If $\dom{\hat{A}} = \ran{a_\emptyset}$ is not dense, then since $\ran{\mbf{A} ^*} = \ran{\mbf{A}^*\mbf{A}}$, we have that there is a non-zero $h \in \H$ so that 
$a_\emptyset h = 0$ (recall that $a_\emptyset \geq 0$ so that $\ran{a_\emptyset} ^\perp = \ker{a_\emptyset}$). It then follows that
$$ \mbf{A} ^* \mbf{A} h = (I - A(0) ^* A(0) - a_\emptyset ^2) h \neq 0, $$ and one cannot have $\mbf{A} ^* \mbf{A} \leq \la ^2 a_\emptyset ^2$ in this case. Hence $\mbf{A} ^* \mbf{A} \leq \la ^2 a_\emptyset ^2$ implies $\dom{\hat{A}}$ has dense range. In this case,
\ba& &  I - a_\emptyset ^2 - A(0) ^* A(0) \leq \la ^2 a_\emptyset ^2 \nn \\
&\Rightarrow  & I - (I + \hat{A} ^* \hat{A} ) ^{-1} - A(0) ^* A(0) \leq \la ^2 (I + \hat{A} ^* \hat{A} ) ^{-1} \nn \\
&\Rightarrow & \frac{1}{1+\la ^2} (I - A(0) ^* A(0) ) \leq (I + \hat{A} ^* \hat{A} ) ^{-1}. \nn \ea 
\end{proof}

\subsection{Clark Intertwining}

Fix $A \in \scr{R} _d (\H)$. As in \cite[Theorem 6.3]{JMfree}, one can verify that the weighted Cauchy transform $\mc{F} ^R : F^2 (\mu _A ) \rightarrow \scr{H} ^R (A)$ intertwines the adjoint of the Stinespring-GNS row isometry $\Pi _A = \pi _A (L)$ with a perturbation of the restricted backward left free shift, $L^* \otimes I_\H | _{\scr{H} ^R (A)} = X$ (and this is a rank-one perturbation in the case where $\H = \C$). Equivalently, the unitary multiplier $\mc{U} _A := M^R _{(I - A ^\dag (Z) )}$ of $\scr{H} ^R _+ (H_A)$ onto $\scr{H} ^R (A)$ intertwines the adjoint of the isometry $V^A$ with a perturbation of $X$: 
\ba  L^* \otimes I_\H | _{\scr{H} ^R (A)} + \mbf{A} (I-A (0) ) ^{-1} (K_0 ^A) ^* & = & \mc{F} ^R \Pi _A ^* (\mc{F} ^R ) ^* \nn \\
& = &  \mc{U} _A (V^A ) ^* \mc{U} _A ^*. \nn \ea
Any $U \in \L (\H )$ yields a different free Clark functional $\mu _{A U^*}$. Since $\scr{H} ^R (A U^*) = \scr{H} ^R (A)$, it follows that every $U \in \L (\H )$ gives a different perturbation of the restricted backward left free shift. In particular, if $A$ is column extreme, each of these perturbations will be a Cuntz unitary (an onto row isometry). In the classical ($d=1$, $\H = \C$) case, one recovers Clark's perturbations of the backward shift.

\bibliography{FreeAC}

\end{document}